\documentclass[final,1p,times]{elsarticle}

\usepackage{graphicx}

\usepackage{amssymb}
\usepackage{amsthm}
\usepackage{amsmath,amssymb,amsopn,amsfonts,mathrsfs,amsbsy,amscd}
\usepackage{longtable}
\usepackage{caption}
\usepackage{multirow}
\usepackage{xcolor}

\newcommand{\prs}{\langle\;,\;\rangle}

\newcommand{\too}{\longrightarrow}
\newcommand{\om}{\omega}
\newcommand{\esp}{\quad\mbox{and}\quad}
\newcommand{\wi}{\widetilde}

\def\br{[\;,\;]}

\newcommand{\G}{\mathfrak{g}}
\newcommand{\g}{\mathfrak{g}}

\newcommand{\h}{{\mathfrak{h}}}
\newcommand{\ad}{{\mathrm{ad}}}
\newcommand{\sd}{{\mathrm{sd}}}

\newcommand{\tr}{{\mathrm{tr}}}

\newcommand{\Li}{{\mathrm{L}}}
\newcommand{\Ri}{{\mathrm{R}}}
\newcommand{\Lei}{\mathrm{Leib}}
\newcommand{\Lie}{\mathrm{Lie}}
\newcommand{\spa}{{\mathrm{span}}}

\newcommand{\Om}{\Omega}

\newcommand{\al}{\alpha}
\newcommand{\be}{\beta}
\newcommand{\ga}{\gamma}

\newcommand{\e}{\epsilon}

\newcommand{\la}{\lambda}
\newcommand{\De}{\Delta}
\newcommand{\de}{\delta}

\newcommand{\Ll}{{\mathrm{L}}}
\newcommand{\Rr}{{\mathrm{R}}}
\newtheorem{Def}{Definition}[section]
\newtheorem{theo}{Theorem}[section]
\newtheorem{pr}{Proposition}[section]

\newtheorem{co}{Corollary}[section]

\newtheorem{exem}{Example}
\newtheorem{remark}{Remark}

\font\bb=msbm10

\def\R{\hbox{\bb R}}

\begin{document}

\begin{frontmatter}
	
	
	
	\title{ Complete Description of Invariant, Associative Pseudo-Euclidean Metrics on Left Leibniz Algebras via Quadratic Lie Algebras  }
	
	
	\author[label1]{ Fatima-Ezzahrae Abid  }
	\address[label1]{Universit\'e Cadi-Ayyad\\
		Facult\'e des sciences et techniques\\
		BP 549 Marrakech Maroc\\e-mail: abid.fatimaezzahrae@gmail.com}
	
	\author[label2]{Mohamed Boucetta}
	\address[label2]{Universit\'e Cadi-Ayyad\\
		Facult\'e des sciences et techniques\\
		BP 549 Marrakech Maroc\\e-mail: m.boucetta@uca.ac.ma}
	
	
	
	
	
	\begin{abstract} A pseudo-Euclidean non-associative algebra $(\G, \bullet)$ is a real algebra of finite dimension that has a metric, i.e., a bilinear, symmetric, and non-degenerate form $\prs$. The metric is considered $\Li$-invariant (resp. $\Ri$-invariant) if all left multiplications (resp. right multiplications) are skew-symmetric. The metric is called associative if $\langle u\bullet v,w\rangle= \langle u,v\bullet w\rangle$ for all $u, v, w \in \g$. These three notions coincide when $\G$ is a Lie algebra and in this case $\G$ endowed with the metric is known as a quadratic Lie algebra.
		
		This paper provides a complete description of $\Li$-invariant, $\Ri$-invariant, or associative pseudo-Euclidean metrics on left Leibniz algebras. It shows that a left Leibniz algebra with an associative metric is also right Leibniz and can be obtained easily from its underlying Lie algebra, which is a quadratic Lie algebra. Additionally, it shows that at the core of a left Leibniz algebra endowed with a $\Li$-invariant or $\Ri$-invariant metric, there are two Lie algebras with one quadratic and the left Leibniz algebra can be built from these Lie algebras. We derive many important results from these complete description.
		Finally, the paper provides a list of left Leibniz algebras with an associative metric up to dimension 6, as well as a list of left Leibniz algebras with an $\Li$-invariant metric, up to dimension 4, and $\Ri$-invariant metric up to dimension 5.
	\end{abstract}

\end{frontmatter}


\section{Introduction}\label{section1}

Lie groups having a bi-invariant pseudo-Riemannian metric is  a  very large class which  contains all semi-simple Lie groups. They became relevant some years ago when they were useful in the formulation of some physical problems, for instance in the so-known Adler-Kostant-Symes scheme. More recently they appeared in conformal field theory \cite{fig}. The Lie algebra of a Lie group having a bi-invariant pseudo-Riemannian is called quadratic. It is a real Lie algebra $(\G,\br)$ endowed with a bilinear non-degenerate form $\prs$ such that, for any $u\in\G$, $\ad_u:\G\too\G$, $v\mapsto [u,v]$ is skew-symmetric with respect to $\prs$. This class of pseudo-Euclidean Lie algebras has been the focus of intensive study by numerous authors (refer to \cite{ovando} and the bibliography therein).

There are three ways of generalizing the notion of bi-invariant pseudo-Euclidean metric on Lie algebras. Let $(\G,\bullet)$ be a nonassociative  algebra and $\prs$ a pseudo-Euclidean metric on $\G$. For any $u\in\G$, we denote by $\Li_u,\Ri_u:\G\too\G$, respectively, the left and the right multiplication by $u$, i.e., $\Li_uv=\Ri_vu=u\bullet v$. The metric is called $\Li$-invariant (resp. $\Ri$-invariant)  if, for any $u\in\G$, $\Li_u$ (resp. $\Ri_u$) is skew-symmetric. It is called associative if, for any $u\in\G$, $\Ri_u^*=\Li_u$, i.e., $\langle u\bullet v,w\rangle=\langle u,v\bullet w\rangle$.
In the first case (resp. the second case), we call $(\G,\bullet,\prs)$ $\Li$-quadratic (resp. $\Ri$-quadratic) and, in the last case, we call $(\G,\bullet,\prs)$ metrised. All these notions coincide when the product $\bullet$ is  anti-commutative.
In the previous decade, various nonassociative algebras endowed with an associative pseudo-Euclidean metric have been studied, for example, general nonassociative algebras (\cite{bor}), quadratic Lie algebras (see \cite{ovando} for a large bibliography), quadratic super-Lie algebras (\cite{al,ben}) or quadratic Malcev super-algebras (\cite{al1}), Leibniz algebras (\cite{lin, said}). However, $\Li$-quadratic or $\Ri$-quadratic nonassociative algebras have attracted little attention. There are some results on $\Li$-quadratic or $\Ri$-quadratic Leibniz algebras in \cite{saidh,lin}.

{Leibniz algebras} are  a non-commutative generalization of Lie algebras and were first introduced and investigated in the papers of  Bloh \cite{bloh, bloh1} under the name of D-algebras.
Then they were  rediscovered by  Loday \cite{loday} who called them  Leibniz algebras. A left Leibniz algebra (resp. right Leibniz algebra) is an algebra $(\g,\bullet)$ over a field $\mathbb{K}$ such that, for any $u\in\G$, $\Li_u$ (resp. $\Ri_u$) is a derivation of $(\G,\bullet)$. An algebra which is both left and right Leibniz is called symmetric Leibniz algebra. A Lie algebra is obviously a symmetric Leibniz algebra.
Many results of the theory of Lie algebras can be extended to left Leibniz algebras (see \cite{ayp, ayp1, ayp2}). Moreover, real left Leibniz algebras are the infinitesimal version of Lie racks. In 2004, Kinyon \cite{kinyon} proved that if $(X, e)$ is a pointed Lie rack, $T_eX$ carries a structure of left Leibniz algebra.

In this paper, we undertake a complete study of $\Li$-quadratic,  $\Ri$-quadratic or metrised left Leibniz algebras.
Our results can be translated to right Leibniz algebras since there is a correspondence between left and right Leibniz algebras. 

Let us explain the two main ideas at the heart of this paper. First, it is known that a metrised left Leibniz algebra is also a right Leibniz algebra and hence it is Lie admissible. We show that the underlying Lie algebra is a quadratic Lie algebra which can be used to build the metrised left Leibniz algebra straightforwardly. 

Secondly, let $(\G,\bullet,\prs)$ be a $\Li$-quadratic or $\Ri$-quadratic left Leibniz algebra. The vector space $\Lei(\G)=\spa\{u\bullet v+v\bullet u\}$ is an ideal known as the Leibniz ideal. Denote by $\Lei(\G)^\perp$ its orthogonal. When $\Lei(\G)\cap\Lei(\G)^\perp=\{0\}$,  $\G$ is a Lie algebra when the metric is $\Ri$-invariant and, when the metric is $\Li$-invariant,
$\G=\Lei(\G)\oplus\Lei(\G)^\perp$ and $\Lei(\G)^\perp$ is a quadratic Lie algebra.

When  $I=\Lei(\G)\cap\Lei(\G)^\perp\not=\{0\}$, 
we show that $\h=\G/I^\perp$  has a structure of  a Lie algebra and $A=I^\perp/I\cap I^\perp$ has a structure of a  $\Li$-quadratic left Leibniz algebra with nondegenerate Leibniz ideal  when the metric is $\Li$-invariant, and a structure of a quadratic Lie algebra when the metric is $\Ri$-invariant . We call $A$ and $\h$ endowed with their structures the core of $(\G,\bullet,\prs)$.  Our main result is that $(\G,\bullet,\prs)$ can be reconstructed from its core in a way which combine the double extension process and the $T^*$-extension. The double extension  is a well-known process which appeared first in \cite{Medina} in the study of quadratic Lie algebras while the $T^*$-extension was introduced  in \cite{bor} in the study of metrised nonassociative algebras.

 Let us enumerate our main results and give the organization of the paper.

\begin{enumerate}
	\item In Section \ref{section2}, we study the properties of $\Li$-invariant, $\Ri$-invariant or associative bilinear forms (not necessary nondegenerate) on left Leibniz algebras. The main result of this section is the introduction of the accurate notion of Killing form on a left Leibniz algebras. It permits the generalization of  Cartan's criterions valid for Lie algebras to left Leibniz algebras (see Propositions \ref{killing} and \ref{sl}).
	\item In Section \ref{section3}, we undertake the study  of metrised left Leibniz algebras. In \cite{said}, it was pointed out that if a left Leibniz algebra has an associative metric then it is also right Leibniz. Based on this and on the characterization of symmetric Leibniz algebras given in \cite{saidbar}, we show that metrised left Leibniz algebras can be obtained from quadratic Lie algebras in an easy way (see Theorems \ref{commutative}, \ref{co} and \ref{dex}). Since quadratic Lie algebras are known up to dimension 6, we get almost without computation all metrised left Leibniz algebras up to dimension 6 (see Tables \ref{1})
	\item In Sections \ref{section4},  \ref{section5} and \ref{section5bis}, we give a complete description of $\Li$-quadratic and $\Ri$-quadratic left Leibniz algebras.  The main results of Sections \ref{section5} and \ref{section5bis} are Theorems \ref{main3} and \ref{main4}  which show that both $\Li$-quadratic and $\Ri$-quadratic left Leibniz algebras can be obtained by a combination of the double extension process and the $T^*$-extension from its core $A$ and $\h$. These two theorems have many interesting corollaries. When $A$ is a Lie algebra with a trivial center and $H^2(A)=0$ then the original algebra can be described easily (See Theorems \ref{semil} and \ref{semir}).
	We derive a complete description of $\Li$-quadratic simple left Leibniz algebras.
	We show that a $\Ri$-quadratic left Leibniz algebra which is semi-simple or for which the metric is definite positive or Lorentzian must be a Lie algebra.   We derive also a complete description of $\Li$-quadratic symmetric Leibniz algebra (see Theorem \ref{main}) and we deduce that any Lorentzian $\Li$-quadratic symmetric Leibniz algebra is actually a Lie algebra.

	\item Section \ref{apend} is an appendix where we give the details of the computations needed in the proofs of Theorems \ref{main3} and \ref{main4}.
	\item Finally, we provide a list of left Leibniz algebras with an associative metric up to dimension 6, as well as a list of left Leibniz algebras with an $\Li$-invariant metric, up to dimension 4, and $\Ri$-invariant metric up to dimension 5 (see Tables \ref{1}-\ref{3}).
\end{enumerate}

\paragraph{Notations} For a non associative algebra $(\G,\bullet)$, we denote by $\Li_u$ and $\Ri_u$, respectively, the left and the right multiplication by $u$, i.e., $\Li_u(v)=u\bullet v$ and $\Ri_u(v)=v\bullet u$. If $F,G$ are two vector subspaces of $\G$, we denote by $F\bullet G$ the subspace generated by $u\bullet v$ for any $u\in G$ and $v\in G$. A vector subspace $I
$ is a left ideal (resp. right ideal) if $\G\bullet I\subset I$ (resp. $I\bullet \G\subset I$) and $I$ is an ideal if it is both a left ideal and a right ideal. Finally, we denote
\[ Z^\ell(\G)=\{u\in \G,\Li_u=0\},\;Z^r(\G)=\{u\in \G,\Ri_u=0\} \esp Z(\G)=Z^\ell(\G)\cap Z^r(\G). \]
A metric on $\G$ is a nondegenerate bilinear symmetric form $\prs$. For any endomorphism $F:\G\too \G$ and a vector subspace $I$, we denote by $F^*$ the adjoint with respect to $\prs$ and by $I^\perp$ the orthogonal with respect to $\prs$. The vector subspace $I$  is called {non-degenerate} if $I \cap I^{\perp}=\{0\}$ and it is called totally isotropic if $I \subset I^{\perp}$. We call a metric Euclidean if it is definite positive and Lorentzian if it has the signature $(n-1,1)$.

\section{Invariant symmetric bilinear forms and the Killing form of left Leibniz algebras}\label{section2}
In the literature, there are three notions of invariant symmetric bilinear forms on a nonassociative algebra. We recall these notions and give some of their properties.
The Killing form of a Lie algebra is an invariant bilinear form which plays an important role and can be used to give a criterion for a Lie algebra to be semi-simple or solvable. Left Leibniz algebras generalize Lie algebras but, as we will show in this section, lack a good notion of Killing form.
We introduce a  bilinear symmetric form on any Leibniz algebra $\G$ and we show that it is the accurate notion of the  Killing form of $\g$ since, in particular, it can give  an analogue of Cartan's criterions for semi-simple and solvable Leibniz algebras.

Let $(A,\bullet)$ be a nonassociative algebra and $S$ a symmetric bilinear form on $A$. $S$ is called
\begin{enumerate}
	\item  $\Li$-invariant if for any $u,v,w\in A$, $S(u\bullet v,w)+S(v,u\bullet w)=0$,
	\item  $\Ri$-invariant if for any $u,v,w\in A$, $S(u\bullet v,w)+S(u,w\bullet v)=0$,
	\item  associative if for any $u,v,w\in A$, $S(u\bullet v,w)=S(u, v\bullet w)$.
\end{enumerate}
It is clear that these three notions coincide if $\bullet$ is anti-commutative. However, when $\bullet$ is commutative and $S$ is $\Li$-invariant or $\Ri$-invariant then $S(u\bullet v,w)=0$ for any $u,v\in A$.

\begin{pr}\label{invariant} Let $(A,\bullet)$ be a nonassociative algebra and $S$ a symmetric bilinear form on $A$. If $S$ satisfies two of the properties above then it satisfies the third one and, for any $u,v\in A$, $u\bullet v+v\bullet u\in\ker S$. In this case, we call $S$ bi-invariant.
	
\end{pr}
\begin{proof} Suppose, for instance, that $S$ is $\Li$-invariant and $\Ri$-invariant. Then, for any $u,v,w\in A$,
	\[ S(u\bullet v,w)=- S(v,u\bullet w)=S(v\bullet w,u) \]and hence $S$ is associative. Moreover,
	\[ S(u\bullet v,w)=-S(v,u\bullet w)=S(v\bullet w,u)=-S(w,v\bullet u) \]which shows that $u\bullet v+v\bullet u\in\ker S$. The other cases can be treated in a similar way.
\end{proof}

Recall that  $(\G,\bullet)$ is a left Leibniz algebra if  for any $u,v,w\in\G$,
\begin{equation}\label{l1}
	u\bullet(v\bullet w)=(u\bullet v)\bullet w+v\bullet (u\bullet w)\quad\mbox{or}\quad \mathrm{L}_{u\bullet v}=[\mathrm{L}_u,\mathrm{L}_v].
\end{equation}

$(\G,\bullet)$ is a right Leibniz algebra if  for any $u,v,w\in\G$,
\begin{equation}\label{R1}
	(v\bullet w)\bullet u=(v\bullet u)\bullet w+v\bullet (w\bullet u)\quad\mbox{or}\quad \mathrm{R}_{w\bullet u }=-[\mathrm{R}_w,\mathrm{R}_u].
\end{equation}
It is obvious that if $(\G,\bullet)$ is a left Leibniz algebra then $(\G,\circ)$ is a right Leibniz algebra when $u\circ v=u\bullet v$.

$(\G,\bullet)$ is called a symmetric Leibniz algebra if it is both a left and right Leibniz algebra.

Let $(\G,\bullet)$ be a	 left Leibniz algebra. We denote $\Lei(\G)=\mathrm{span}\left\{ u\bullet v+v\bullet u,u,v\in\G  \right\}$. 
It is well-known \cite{ay} that $\Lei(\G)$ is an ideal of $\G$, for any $u\in\Lei(\G)$, $\mathrm{L}_u=0$, $\Lie(\G)=\G/\Lei(\G)$ is a Lie algebra and $\G$ is a Lie algebra if and only if $\Lei(\G)=\{0\}$. We call $\Lei(\G)$ the Leibniz ideal of $\G$.

We consider $B$ and $K$ the two bilinear symmetric forms on $\G$  given by  $$B(u,v)=\tr(\mathrm{L}_u\circ\mathrm{L}_v)
\esp K(u,v)=-\frac12\left(\tr(\mathrm{L}_{u}\circ \mathrm{R}_{v}) +\tr(\mathrm{L}_{v}\circ \mathrm{R}_{u})  \right).
$$ 

The following proposition is our first main result.

\begin{pr}\label{killing} Let $(\G,\bullet)$ be a left Leibniz algebra. Then the 2-forms $B$ and  $K$ are bi-invariant and hence $\ker B$ and $\ker K$ contain $\Lei(\G)$. Moreover, for any $u,v\in \G$,
	\begin{equation} \label{k}
		K(u,v)=\widehat{K}(\pi(u),\pi(v)) \end{equation}
	where $\pi:\G\too \G/\Lei(\G)$ and $\widehat{K}$ is the Killing form of the Lie algebra $\G/\Lei(\G)$. Furthermore, $\ker K$ is a solvable ideal of $\G$ and hence its contained in the radical of $\G$ (the largest solvable ideal in $\G$).

\end{pr}

\begin{proof} The fact that $B$ is bi-invariant is an immediate consequence of \eqref{l1}. To show that $K$ is bi-invariant, it suffices to show \eqref{k}.
	Consider a complement $V$ of $\Lei(\G)$, i.e., $\G=\Lei(\G)\oplus V$ and put, for any $u,v\in\G$, 
	\[ u\bullet v=[u,v]_V+\om(u,v) \]where $[u,v]_V\in V$ and $\om(u,v)\in\Lei(\G)$. Since, for any $u,v\in\G$, $u\bullet v+v\bullet u\in \Lei(\G)$ then $\br_V$ is skew-symmetric and, by using the Leibniz identity,
	one can check easily that $(V,\br_V)$ is a Lie algebra. For any $u\in V$, denote by $\ad_u^V:V\too V$ the adjoint. For any $u\in\Lei(\G)$ and $v\in V$, 
	$$\mathrm{L}_u=0,\;  \mathrm{R}_u=\left(\begin{array}{cc}0&\om(.,u)\\0&0 \end{array}     \right),\;
	\mathrm{L}_v=\left(\begin{array}{cc}D_v&\om(v,.)\\0&\ad_v^V \end{array}     \right)\esp \mathrm{R}_v=\left(\begin{array}{cc}0&\om(.,v)\\0&-\ad_v^V \end{array}     \right).
	$$ So
	\[ K(u,v)=\begin{cases} 0\quad\mbox{if} \quad u\in \Lei(\G)\;\mbox{or}\; v\in \Lei(\G),\\
		\tr_V(\ad_u^V\circ\ad_v^V)=K_V(u,v)\quad\mbox{if} \quad u,v\in V,
	\end{cases} \]where $K_V$ is the Killing form of $(V,\br_V)$. From this formula, we can deduce that 
	$\pi:V\too\G/\Lei(\G)$ is an isomorphism of Lie algebras which sends $K_V$ to the quotient of $K$.
	
	Since $K$ is bi-invariant then $\ker K$ is obviously an ideal.  We have $\Lei(\G)$ is solvable and, according to \eqref{k}, $\ker K/\Lei(\G)=\ker\widehat{K}$. It is a well-known  fact \cite{knapp} that the kernel of the Killing form of a Lie algebra is solvable,  hence $\ker K$ is solvable.
\end{proof}

The 2-form $B$ was presented by many authors as the Killing form of the Leibniz algebra (see \cite{ay, aypl}). However, it presents a serious inconvenient, namely, the quotient $\wi B$ of $B$ on $\G/\Lei(\G)$ is not, in general, the Killing form of  $\G/\Lei(\G)$ as the following example shows.

\begin{exem}\label{exem1} Consider the 2-dimensional Leibniz algebra $\G=\R^2$ endowed with the bracket
	\[ [e_2,e_1]=e_1\esp [e_2,e_2]=e_1. \]
	Then $B(e_1,e_1)=B(e_1,e_2)=0$ and $B(e_2,e_2)=1$. The quotient $\G/\Lei(\G)=\R e_2$ is abelian but the quotient of $\wi B$ of $B$ satisfies $\wi B(e_2,e_2)=1$. However, $K=0$.
	
\end{exem}

\begin{remark} From the proof of the last proposition, one can see that, for any $u,v\in V$,
	\[ B(u,v)=\tr_{\Lei(\G)}(D_u\circ D_v)+\tr_V(\ad_u^V\circ\ad_v^V) \]which explains the difference between $B$ and $K$.
\end{remark}

It is known \cite{ay} that a left Leibniz algebra is semi-simple (resp. solvable) if and only if $\G/\Lei(\G)$ is semi-simple (resp. solvable). Thus we can deduce the following result which generalize Cartan's criterions and shows that $K$ is the good notion of Killing form.

\begin{pr}\label{sl}\begin{enumerate}\item  A left Leibniz algebra $\G$ is semi-simple if and only if $\ker K=\Lei(\G)$. 
		\item  A left Leibniz algebra $\G$ is solvable if and only if $\G\bullet \G\subset\ker K$.
	\end{enumerate}
\end{pr}

\section{Metrised   left Leibniz algebras}\label{section3}

Associative metrics on left Leibniz algebra were studied in \cite{said, saidh}. It was shown that if a left Leibniz algebra has an associative metric then it is also right Leibniz. In this section, we complete the results obtained in \cite{said, saidh} by giving a simple and complete description of symmetric Leibniz algebras having an associative metric. We call such pseudo-Euclidean symmetric Leibniz algebras metrised.

Let $(\G,\bullet)$ be a nonassociative algebra. Denote, respectively, by $\br$ and $\circ$ the skew-symmetric and symmetric parts of $\bullet$. That is
\begin{equation}\label{pralg}
	[u,v]= \frac12(u\bullet v- v\bullet u), \; u \circ v = \frac12(u\bullet v+v \bullet u) ,\esp u \bullet v=[u,v]+u\circ v. \; \forall u,v \in \g.
\end{equation}
Put, $\ad_u: \g \too \g$, $ v\mapsto [u,v]$ and $\sd: \g \too \g$,$ v\mapsto u \circ v$ for any $u \in \g$. Hence, 
\[ \Ll_u=\ad_u+\sd_u\esp \Rr_u=-\ad_u+\sd_u.\]

The following  result, even if it is easy to establish, will have an important consequence and to our knowledge doesn't appear in the literature.
\begin{pr}\label{assoc}
	Let $(\G,\bullet,\prs)$ be an algebra endowed with a metric $\prs$. Then $\prs$ is   associative if and only if $\prs$ is bi-invariant on
	$(\G,\br)$ and associative on  $(\G,\circ)$.
\end{pr}
\begin{proof}
	The metric is associative if and only if $\Ll_{u}=\Rr_{u}^{*}$ for any $u\in g$. That is,  for any $u\in\G$,
	\[  \ad_u+\mathrm{sd}_u=
	-\ad_u^*+\mathrm{sd}_u^*.\]
	Thus, $\ad_u+\ad_u^*=\sd_u^*-\sd_u$. Since $\ad_u+\ad_u^*$ is  symmetric  and $\sd_u^*-\sd_u$   skew-symmetric  then $\ad_u+\ad_u^*=\sd_u^*-\sd_u=0$, for any $u \in \g$  which completes the proof.
\end{proof}

Let $(\g,\bullet,\prs)$ be  a left Leibniz algebra endowed with an associative metric.  Then, for any $u\in \g$, $\Li_u^*=\Ri_u$ and by using \eqref{l1} and $(2)$, we deduce
 that $(\g,\bullet)$ is  also right Leibniz algebra. This remark appeared in \cite[Theorem 2.4]{said}.   
We conclude that if a left Leibniz algebra has an associative metric then it is a symmetric Leibniz algebra.

On the other hand, according to \cite[Proposition 2.11]{saidbar}, any symmetric Leibniz algebra $(\G,\bullet)$ is given by 
\[ u\bullet b=[a,b]+\om(u,v), \]where $\br$ is a Lie bracket on $\G$ and $\om$ is a bilinear symmetric map  
$\omega \, : \g \times \g \longrightarrow Z(\g)$ ($Z(\G)$ is the center of $(\G,\br)$), $(x,y) \mapsto \omega(x,y)= x\circ y$ such that, for any $x,\, y,\, z \in \g$,
\begin{equation}\label{eq}
	\omega([x, y], z) = \omega(\omega(x, y), z) = 0.
\end{equation}
We call $(\G,\br+\om)$ the symmetric Leibniz algebra obtained from the Lie algebra $(\G,\br)$ by means of $\om$. 

By using this characterization of symmetric Leibniz algebras an  Proposition \ref{assoc}, we get a new and complete description of symmetric Leibniz algebras endowed with an associative metric. But before, we can give a complete description of metrised symmetric Leibniz algebras whose underlying Lie algebra is abelian.

\begin{theo} \label{commutative} Let $(\G,\bullet,\prs)$
	be a metrised 	  commutative symmetric Leibniz algebra. Then there exists a pseudo-Euclidean vector space $(A,\prs_A)$, a vector space $\h$ and a symmetric 3-linear form $T:\h\times\h\times\h\too\R$ such that $(\G,\bullet,\prs)$ is isomorphic to $(\h\oplus A\oplus \h^*,\circ,\prs_0)$ where   $\h\circ\h\subset\h^*$ and for any $X,Y,Z\in\h$, $a\in A$, $\al\in\h^*$,
	\[ \langle X+a+\al,X+a+\al\rangle_0=2\prec\al,X\succ+\langle a,a\rangle_A,\;\Li_a=\Li_\al=0  \esp \prec X\circ Y,Z\succ=T(X,Y,Z), \]

\end{theo}
\begin{proof}
	According to what above, a commutative symmetric Leibniz algebra is a commutative algebra $(\G,\bullet)$ such that for any $u,v\in\G$, $\Li_u\circ\Li_v=0$. Suppose that $\prs$ is an associative metric on $\G$. Then, for any $u\in\G$, $\Li_u$ is symmetric with respect to $\prs$ and, if $I=\sum_{u\in\G}\mathrm{Im}\Li_u$, $I^\perp=\bigcap_{u\in\G}\ker\Li_u$ and $I\subset I^\perp$. Write $I^\perp=I\oplus A$ for a complement $A$. The restriction $\prs_A$ of the metric to $A$ is nondegenerate and  there exists a totally isotropic space $\h$ such that $A^\perp=I\oplus\h$ and $\G=I\oplus A\oplus \h$. The metric identify $I$ to the dual $\h^*$ and we can identify $(\G,\prs)$ to $(\h^*\oplus A\oplus\h,\prs_0)$.
	
	With this identification in mind, the bracket $\bullet$ satisfies $\h^*\oplus A\subset Z(\h^*\oplus A\oplus\h)$ and $\h\bullet\h\subset \h^*$ and the 3-linear form given by
	\[ T(X,Y,Z):=\prec X\bullet Y,Z\succ=\prec Y,X\bullet Z\rangle, \]for any $X,Y,Z\in\h$ is symmetric and defines $\bullet$.\end{proof}

We state now an important result.

\begin{theo}\label{co} Let $(\G,\bullet,\prs)$ be a pseudo-Euclidean symmetric Leibniz algebra obtained from $(\G,\br)$ by means of $\om$. We denote by $Z(\G)$ the center of $(\G,\br)$.  Then $\prs$ is associative if and only if  $(\G,\br,\prs)$ is a quadratic Lie algebra and there exists a 3-linear symmetric form $T:\G\times\G\times\G\too\R$ and a totally isotropic vector subspace $I\subset Z(\G)$ such that $T(I^\perp,.,.)=0$ and 
	\begin{equation}\label{T} \langle\om(x,y),z\rangle =T(x,y,z),\quad x,y,z\in\G. 
	\end{equation}
\end{theo}

\begin{proof} According to Proposition \ref{assoc} and the description of symmetric Leibniz algebras given above, if $\prs$ is associative then $(\G,\br,\prs)$ is a quadratic Lie algebra and $\om$ satisfies $I=\mathrm{span}\{\om(x,y),x,y\in\G\}\subset Z(\G)$, 
	\[ \om([x,y],z)=\om(\om(x,y),z)=0\esp \langle \om(x,y),z\rangle=\langle \om(x,z),y\rangle,\quad x,y,z\in\G.  \]
	For any $x,y,z,t\in\G$,
	\[ 0= \langle \om(\om(x,y),z),t\rangle =\langle \om(z,t),\om(x,y)\rangle =0\]and hence $I$ is totally isotropic. 
	Put $T(x,y,z)= \langle \om(x,y),z\rangle$. Then $T$ is symmetric and for any $z\in I^\perp$ and $x,y\in\G$,
	\[ T(x,y,z)=\langle\om(x,y),z\rangle =0. \]
	
	Conversely, let $(\G,\br,\prs)$ be a quadratic Lie algebra, $I\subset Z(\G)$ a totally isotropic vector subspace and $T:\G\times\G\times\G\too\R$ a
	3-linear symmetric form such that $T(I^\perp,.,.)=0$. Define the symmetric bilinear form $\om:\G\times\G\too\G$ by
	\[ T(x,y,z)=\langle\om(x,y),z\rangle =0,\;x,y,z\in\G. \]It is obvious that $\om$ takes its values in $I$ and, for any $x,y,z,t\in\G$,
	\[ \langle \om(\om(x,y),z),t\rangle=T(\om(x,y),z,t)=0\esp \langle \om([x,y],z),t\rangle=\langle\om(z,t),[x,y]\rangle=-\langle[x,\om(z,t)],y\rangle=0. \]
	So $(\G,\br+\om,\prs)$ is a metrised symmetric Leibniz algebra.
\end{proof}

\begin{remark}

	Let $(\G,\bullet,\prs)$
	be a metrised 	  symmetric Leibniz algebra obtained from $(\G,\br)$ by means of $\om$. If $[\G,\G]=\G$ then $Z(\G)=[\G,\G]^\perp=\{0\}$. Thus $\om=0$ and hence  $(\G,\bullet,\prs)$ is a quadratic Lie algebra.

\end{remark}

In Theorem \ref{co}, the determination of $I$ and $T$ is an easy task so  the determination of metrised non Lie symmetric Leibniz algebras reduces to the determination of quadratic Lie algebras with non trivial center. If the center is nondegenerate, we will see that they can be described easily. Moreover,
quadratic Lie algebras with non trivial degenerate center are determined by the double extension process (see \cite{ovando, favre}).
Let us recall the process of double extension (see \cite{ovando} for more details).

Let $(\h,\br_\h,\prs_\h)$ be a quadratic  Lie algebra. The double extension of $\h$ by  means of $A\in\mathrm{so}(\h,\prs_\h)$ is the Lie algebra $(\de_A(\h)=\R e\oplus\h\oplus\R \bar{e},\br)$ where the non vanishing Lie brackets and the metric are given, for any $u,v\in\h$, by
\[ [\bar{e},u]=Au,\; [u,v]=\langle Au,v\rangle_\h e+[u,v]_\h,\; \langle xe+u+\bar{x}\bar{e},xe+u+\bar{x}\bar{e}\rangle =2x\bar{x}+\langle u,u\rangle_\h.  \]
$(\de_A(\h),\br,\prs)$ is a quadratic Lie algebra.

\begin{theo}\label{dex} Let $(\G,\bullet,\prs)$ be a metrised symmetric Leibniz algebra obtained from $(\G,\br)$ by means of $\om$ such that $[\G,\G]\not=\G$. Then the following assertions hold.
	
	\begin{enumerate}
		\item If $Z(\G)$ is non-degenerate then $\G=[\G,\G]\oplus Z(\G)$ and there exists a totally isotropic subspace $I\in Z(\G)$ and a 3-linear form $T :Z(\G)\times Z(\G)\times Z(\G)\too \R$ such that $T(I^\perp\cap Z(\G),.,.)=0$ and
		\[ x\bullet y=\begin{cases} [x,y]\quad\mbox{if}\quad x\in[\G,\G]\;\mbox{or}\; y\in[\G,\G],\\
			\om(x,y)\quad\mbox{if}\quad x,y\in Z(\G).
		\end{cases} \]where the restriction of $\om$ to $Z(\G)$ is given by \eqref{T}.
		\item If $Z(\G)$ is degenerate then $(\G,\br,\prs)=\de_A(\h)$ and when $A$ is invertible or $(\h,\br_\h,\prs_\h)$ is Euclidean then $\om$ satisfies $\om(\h\oplus\R e,.)=0$ and $\om(\bar{e},\bar{e})=\mu e$ with $\mu\in\R$.
		
	\end{enumerate}

\end{theo}

\begin{proof} The first assertion is an immediate consequence of Theorem \ref{co} and the fact that $[\G,\G]^\perp=Z(\G)$. 
	
	If If $Z(\G)$ is degenerate then $(\G,\br,\prs)$ is obtained by the double extension process and hence $(\G,\br,\prs)=\de_A(\h)$. In this case, $Z(\G)=\R e\oplus Z(\h)\cap \ker A$. If $A$ is invertible then $Z(\G)=\R e$ and we get the result as consequence of Theorem \ref{co}. If the metric on $\h$ is definite positive then $\R e$ is  the only non trivial totally isotropic subspace of $Z(\G)$ and the result follows.
	\end{proof}

\begin{co}\label{definite} Let $(\G,\bullet,\prs)$ be a metrised symmetric Leibniz algebra such the restriction of $\prs$ to  $Z(\G)$ is definite positive. Then  $(\G,\bullet,\prs)$ is a quadratic Lie algebra.
	
\end{co}

The lists of indecomposable quadratic Lie algebras of dimension $\leq6$, up to an automorphism, are given in \cite{ovando} and \cite{Baum}. From these lists, we can find all quadratic Lie algebras of dimension $\leq6$, up to an automorphism. By using Theorems \ref{commutative}-\ref{dex} and Corollary \ref{definite}, we get  almost without any computation the following result.

\begin{theo}\label{main2}\begin{enumerate}\item
		In dimension 2 there is only one metrised non Lie symmetric Leibniz algebra up to an isomorphism, namely $\R^2$ endowed with the bracket and the metric
		\[ e_1\bullet e_1=\pm e_2\esp m=\left(\begin{matrix}
			0&\al\\\al&0
		\end{matrix}\right),\quad \al>0. \]
		\item	In dimension 3 there is only one metrised non Lie symmetric Leibniz algebra up to an isomorphism, namely $\R^3$ endowed with the bracket and the metrics
		\[ e_1\bullet e_1=\pm e_3\esp m=\left(\begin{matrix}
			0&\al&0\\\al&0&0\\0&0&\pm1
		\end{matrix}\right)
		,\quad\al>0. \]
		\item	In dimension 4,5 and 6, up to an isomorphism, all
		symmetric Leibniz algebra carrying an associative metric and obtained from an indecomposable quadratic Lie algebra are given in \ref{1}.
	\end{enumerate}
	
\end{theo}

\section{$\Li$-quadratic and  $\Ri$-quadratic   left Leibniz algebras}\label{section4}

In this section, we give some immediate properties of $\Li$-quadratic and  $\Ri$-quadratic   left Leibniz algebras, we describe $\Li$-quadratic left Leibniz algebras with nondegenerate Leibniz ideal and we define the core of $\Li$-quadratic and  $\Ri$-quadratic   left Leibniz algebras.

\subsection{General properties of $\Li$-quadratic and  $\Ri$-quadratic   left Leibniz algebras}

A $\Li$-quadratic (resp. $\Ri$-quadratic) left Leibniz algebra is a pseudo-Euclidean left Leibniz algebra $(\G,\bullet,\prs)$ such that the metric is $\Li$-invariant (resp. $\Ri$-invariant). This means that for any $u\in\G$, $\Li_u^*=-\Li_u$ (resp. $\Ri_u^*=-\Ri_u$). 

According to Proposition \ref{invariant}, if a pseudo-Euclidean left Leibniz algebra is both $\Li$-quadratic and $\Ri$-quadratic then it is a Lie algebra.

The following example shows that the classes of $\Li$-quadratic or $\Ri$-quadratic  left Leibniz algebras are large. 
\begin{exem}\label{exem2}\begin{enumerate}\item Let $(\G,\br)$ be a Lie algebra. We denote by $\mathrm{ad}^*:\G\too\mathrm{End}(\G^*)$ its co-adjoint representation, i.e., $\prec \mathrm{ad}^*_u\al,v\succ=-\prec\al,[u, v]\succ$. We endow $T^*\G=\G\oplus\G^*$ by the bracket
		\[ (u+\al)\bullet(v+\be)=[u,v]+\mathrm{ad}^*_u\be,\quad u,v\in\G,\al,\be\in\G^* \]
		and its canonical neutral metric given by
		\[ \langle u+\al,v+\be\rangle=\prec\al,v\succ+\prec\be,u\succ. \] 
		Then $(T^*\G,\bullet,\prs)$ is a $\Li$-quadratic left Leibniz algebra. 
		
		\item Let, $(\G,\br)$ be a 2-step nilpotent Lie algebra. . We endow $T^*\G=\G\oplus\G^*$ by the bracket
		\[ (u+\al)\bullet(v+\be)=\mathrm{ad}^*_u\be,\quad u,v\in\G,\al,\be\in\G^* \]
		and its canonical neutral metric $\prs$. Then $(T^*\G,\bullet,\prs)$ is a $\Ri$-quadratic left Leibniz algebra.

	\end{enumerate}
	
\end{exem}

The following  proposition  can be easily established.
\begin{pr}\label{pr}Let $(\g,\bullet, \prs)$ be a pseudo-Euclidean left Leibniz algebra and $I$ an ideal of $\G$. Then the following assertions hold.
	\begin{enumerate}
		\item If  $\prs$ is $\Ll$-invariant $($resp.$\Rr$-invariant) then $Z^r(\g)=(\g\bullet \g)^\perp$ (resp.
		$Z^l(\g)=(\g\bullet \g)^\perp)$.
		\item If $\prs$ is $\Ll$-invariant, then $ I^{\perp}$ is a left ideal of $\g$ and $I\bullet  I^{\perp}=0$. Moreover, $I^{\perp}$ is an ideal of $\g$ if and only if  $ I^{\bot}\bullet  I=0$.
		\item If  $\prs$ is $\Rr$-invariant, then $ I^{\perp}$ is a right ideal of $\g$ and $I^{\perp}\bullet I=0$. Moreover, $ I^{\perp}$ is an ideal of $\g$ if and only if  $ I\bullet I^{\perp}=0$.

	\end{enumerate}
\end{pr}

The following proposition will be useful later.
\begin{pr}\label{leib}Let $(\g,\bullet, \prs)$ be a pseudo-Euclidean left Leibniz algebra. Then the following assertions hold.
	\begin{enumerate}
		\item If  $\prs$ is $\Ll$-invariant  then
		$
		\Lei(\g)^\perp=\{ u\in\G, \mathrm{R}_u +\mathrm{R}_u^*=0 \}. 
		$ Moreover, $\Lei(\g)^\perp$ is a left ideal,  $I=\Lei(\G)+\Lei(\g)^\perp$ is an ideal and $I\bullet I^\perp=0$
		\item If  $\prs$ is $\Rr$-invariant  then
		$
		\Lei(\G)^\perp=\{ u\in\G, \mathrm{L}_u +\mathrm{L}_u^*=0 \}.
		$
		Moreover, $\Lei(\G)\subset \g\bullet \g \subset \Lei(\G)^\perp$. Thus, $\Lei(\G)$ is a totally isotropic  ideal of $\g$,  $\Lei(\G)^\perp$ is an ideal and $\Lei(\G)^\perp\bullet \Lei(\G)=0$.
	\end{enumerate}
\end{pr}
\begin{proof}\begin{enumerate}\item
	Let $u\in \Lei(\G)^\perp$. We have, for any $v,w\in\G$,
	$$
	0=\langle u,v\bullet w+w\bullet v\rangle\\
	=-\langle v\bullet u,w\rangle-\langle w\bullet u,v\rangle\\
	=-\langle\Ri_uv,w\rangle-\langle\Ri_uw,v\rangle
	$$and hence $\Ri_u^*=-\Ri_u$. According to Proposition \ref{pr}, $\Lei(\G)^\perp$ is a left ideal. Put $I=\Lei(\g)+\Lei(\g)^\perp$. Then $I$ is a left ideal,  $I^\perp=\Lei(\g)\cap\Lei(\g)^\perp$ and, for any $u\in\G$, $v\in I$ and $w\in I^\perp$,
	\[ \langle v\bullet u,w\rangle =-\langle u,v\bullet w\rangle=\langle u\bullet w,v\rangle=0 \]since $I^\perp$ is a left ideal and $\Ri_w^*=-\Ri_w$. Hence $I$ is an ideal and $I\bullet I^\perp=0$, according to Proposition \ref{pr}.
	
\item 	The first part of the second assertion can be obtained similarly. Since $\Lei(\G)$ is an ideal, $\Lei(\G)^\perp\bullet\Lei(\G)=0$ by virtue of Proposition \ref{pr}.

On the other hand, recall that  for any $u\in\Lei(\G)$, $\Li_u=0$. 
Then  for $v,w\in\G$ and $u\in\Lei(\G)$,
	\[ \langle u,v\bullet w\rangle=-\langle u\bullet w,v\rangle=0 \]so $\Lei(\G)\subset(\G\bullet\G)^\perp$ which completes the proof since obviously $\Lei(\G)\subset \G\bullet\G$.\qedhere
\end{enumerate}
\end{proof}
The following result is an immediate consequence of Proposition \ref{leib}.
\begin{pr} A $\Ri$-quadratic Euclidean left Leibniz algebra is necessarily a Lie algebra.
	
\end{pr}

\subsection{$\Li$-quadratic left Leibniz algebras with nondegenerate Leibniz ideal}

$\Li$-quadratic left  Leibniz algebras with nondegenerate Leibniz ideal are quite easy to describe. We call them nondegenerate 
$\Li$-quadratic left Leibniz algebras. 

Let $(\G,\bullet,\prs)$ be a $\Li$-quadratic left Leibniz algebra such that $\Lei(\G)$ is nondegenerate. According to Proposition \ref{leib}, 
$\g= \Lei(\g) \oplus \Lei(\g)^{\perp}$, where $(\Lei(\g)^{\perp},\bullet,\prs)$ is a quadratic Lie algebra and there is a representation $\rho : \Lei(\g)^{\perp} \too \mathrm{so}(\Lei(\g),\prs)$ given by $\rho(u)(v)=u\bullet v$.

Conversely, let $(\h,\prs_\h)$ be  a quadratic Lie algebra,  $(N,\prs_N)$ a pseudo-Euclidean vector space and $\rho:\h\too \mathrm{so}(N,\prs_N)$ a representation of $\h$. We define on $\G=\h\oplus N$ the product
\[ u\bullet v=\begin{cases}0,\;\mbox{if}\; u\in N,\\
	[u,v]_\h\; \mbox{if}\; u,v\in \h,\\
	\rho_u(v)\;\mbox{if}\; u\in \h,v\in N,
\end{cases} \]and the metric $\prs=\prs_\h\oplus\prs_N$. Then $(\G,\bullet,\prs)$ is a $\Li$-quadratic left Leibniz algebra.

\subsection{The core of $\Li$-quadratic and $\Ri$-quadratic left Leibniz algebras}

In this subsection, we construct two associated algebras for any $\Li$-quadratic (resp. $\Ri$-quadratic) left Leibniz algebra: a Lie algebra and a nondegenerate $\Li$-quadratic left Leibniz algebra (resp. a quadratic Lie algebra). We call these algebras  {\it the core} of the $\Li$-quadratic (resp. $\Ri$-quadratic) left Leibniz algebra and they will be used in Sections \ref{section5} and \ref{section5bis} to reconstruct the original algebra.

\begin{pr}\label{core} Let $(\G,\bullet,\prs)$ be a $\Li$-quadratic (resp. $\Ri$-quadratic) left Leibniz algebra and denote by $I=\Lei(\G)\cap \Lei(\G)^\perp$ (resp. $I=\Lei(\G)$), $A=I^\perp/I$ and $\h=\G/I^\perp$. Let $\pi_A:I^\perp\too A$ and $\pi_\h:\G\too \h$
	the canonical projections.
	Then the following assertions hold.
	\begin{enumerate}
		\item If the metric is $\Li$-quadratic then $A$ has a structure of $\Li$-quadratic left Leibniz  algebra given by
		\[ \pi_A(u)\bullet_A\pi_A(v)=\pi_A(u\bullet v)\esp \langle \pi_A(u),\pi_A(v)\rangle_A=\langle u,v\rangle \]and $\Lei(A)=\pi_A(\Lei(\G))$ is nondegenerate and $$A=\Lei(A)\stackrel{\perp}\oplus \pi_A(\Lei(\G)^\perp)$$ where $\pi_A(\Lei(\G)^\perp)$ is a quadratic Lie algebra.
		\item If the metric is $\Ri$-quadratic then $A$ has a structure of  a quadratic Lie  algebra given by
		\[ [\pi_A(u),\pi_A(v)]_A=\pi_A(u\bullet v)\esp \langle \pi_A(u),\pi_A(v)\rangle_A=\langle u,v\rangle. \]
		
		\item If the metric is $\Li$-invariant then $I^\perp$ is an ideal of $\G$ and $\h=\G/I^\perp$ has a structure $\br_\h$ of a Lie algebra.
		
		\item  If the metric is $\Ri$-invariant then the bracket $\br_\h$ on $\h$ given by the formula
		\begin{equation}\label{sub} [\pi_\h(u),\pi_\h(v)]_\h=-\pi_\h(\Li_v^*u) \end{equation}
		is well-defined and it is a Lie bracket and
		 for any $X,Y,Z\in\h$,
		 \[ [[X,[Y,Z]_\h]]_\h=0. \]	
	\end{enumerate}
	
\end{pr}
\begin{proof} 
	\begin{enumerate}
		\item According to Proposition \ref{leib}, $\Lei(\G)$ is an ideal and $\Lei(\G)^\perp$ is a left ideal and hence  $I^\perp$ is a subalgebra of $\G$ and $I$ is an ideal in $I^\perp$. This implies that $A$ has a structure of left Leibniz algebra and the other claims follow immediately.
		\item According to Proposition \ref{leib}, both $\Lei(\G)$ and $\Lei(\G)^\perp$ are ideal and $\Lei(\G)\subset \Lei(\G)^\perp$ this shows that $A$ has a structure of left Leibniz algebra. Moreover, for any $u\in\Lei(\G)^\perp$, $u\bullet u\in\Lei(\G)$ and hence $\pi_A(u\bullet u)=0$ which shows that $A$ is actually a Lie algebra.
		\item We have seen in Proposition \ref{leib} that $I^\perp$ is an ideal and hence $\G/I^\perp$ has a structure of left Leibniz algebra. Moreover, for any $u\in\G$, $u\bullet u\in\Lei(\G)\subset I^\perp$ and hence $\pi_\h(u\bullet u)=0$ so the quotient is actually a Lie algebra.
		\item We have seen in Proposition \ref{leib} that $\Lei(\G)^\perp$ is an ideal which contains $\G\bullet\G$ and hence the quotient $\G/\Lei(\G)^\perp$ is trivial. However, we will show that the bracket $\br_\h$ is a Lie bracket which is not, in general,  the quotient bracket.
		
		Note first that for any $u,v,w\in\G$,
		\[ \langle \Li_u^*v,w\rangle =\langle v,u\bullet w\rangle =-\langle v\bullet w,u\rangle=-\langle \Li_v^*u,w\rangle \]and hence $\Li_u^*v=-\Li_v^*u$. Moreover, for any $u\in I^\perp$, $v\in\G$ and $w\in I$,
		\[ \langle \Li_u^*v,w\rangle=\langle v,u\bullet w\rangle=-\langle v\bullet w,u\rangle=0 \]since $I=\Lei(\G)$ is an ideal. So we have shown that $\br_\h$ is well-defined and skew-symmetric.
		
		Let $X=\pi_\h(u),Y=\pi_\h(v),Z=\pi_\h(w)$. We have
		\[ [X,[Y,Z]_\h]_\h=\pi_\h(\Li_u^*\circ\Li_v^*w). \]
		By virtue of \eqref{l1},
		\[ \Li_u^*\circ\Li_v^*w=\Li_v^*\circ\Li_u^*w-\Li_{v\bullet u}^*w. \]
		But for any $t\in I=\Lei(\G)$, we have
		\[ \langle\Li_{v\bullet u}^*w,t\rangle=\langle w,(v\bullet u)\bullet t\rangle=-\langle w\bullet t,v\bullet u\rangle=0 \]since we have seen in Proposition \ref{leib} that $I$ is an ideal and $I\subset(\G\bullet\G)^\perp$. So we get that for any $X,Y,Z\in\h$,
		\[ [X,[Y,Z]_\h]_\h=[Y,[X,Z]_\h]_\h. \]
		This relation combined to the skew-symmetry of $\br_\h$ will show that $[X,[Y,Z]_\h]_\h=0$. Indeed,
		\begin{align*}
			[X,[Y,Z]_\h]_\h&=[Y,[X,Z]_\h]_\h\\
			&=-[Y,[Z,X]_\h]_\h\\
			&=-[Z,[Y,X]_\h]_\h\\
			&=[Z,[X,Y]_\h]_\h\\
			&=[X,[Z,Y]_\h]_\h\\
			&=-[X,[Y,Z]_\h]_\h
		\end{align*}which completes the proof.\qedhere
		
	\end{enumerate}

\end{proof}

\begin{Def}\label{def} Let $(\G,\bullet,\prs)$ be a $\Li$-quadratic  (resp. $\Ri$-quadratic) left Leibniz algebras. We call the $\Li$-quadratic left Leibniz algebra $(A,\bullet_A,\prs_A)$ and the Lie algebra $(\h,\br_\h)$ (resp. the quadratic Lie algebra $(A,\br_A,\prs_A)$ and the Lie algebra $(\h,\br_\h)$)  defined in Proposition \ref{core} the core of $(\G,\bullet,\prs)$.
	
\end{Def}

Actually, we can go further and give more information about the core of a $\Li$-quadratic or $\Ri$-quadratic left Leibniz algebra.

\begin{pr}\label{spliting}
 Let $(\G,\bullet,\prs)$ be a pseudo-Euclidean left Leibniz algebra. Denote by $\Lie(\G)=\G/\Lei(\G)$ the quotient Lie algebra and $\pi:\G\too \Lie(\G)$ the canonical projection.

\begin{enumerate}
	\item If $(\G,\bullet,\prs)$ is $\Li$-quadratic then
	the map $\kappa:A\too\Lie(\G)$ is a well-defined morphism of left algebras, $\ker\kappa=\pi_A(\Lei(\G))$ and $\pi_A(\Lei(\G)^\perp)$ with an ideal of $\Lie(\G)$ and the Lie algebra quotient $\Lie(\G)/\pi_A(\Lei(\G)^\perp)$ is isomorphic to $(\h,\br_\h)$
	\item If $(\G,\bullet,\prs)$ is $\Ri$-quadratic then the map $\kappa:A\too\Lie(\G)$ is a well-defined into morphism of Lie algebras and $\kappa(A)$ is an ideal of $\Lie(\G)$ which contains $[\Lie(\G),\Lie(\G)]$. 
\end{enumerate}

\end{pr}
\begin{proof} \begin{enumerate}
		\item The first part of statement is obvious and the second part is a consequence of the fact that the map $m:\Lie(\G)\too\h$, $\pi(u)\mapsto\pi_\h(u)$ is an onto morphism of Lie algebras and $\ker m=\pi_A(\Lei(\G)^\perp)$.
		\item It is clear if one has in mind the fact $\Lei(\G)\subset\g\bullet\g\subset\Lei(\G)^\perp$ established in Proposition \ref{leib}.\qedhere
	\end{enumerate}
	
\end{proof}

Recall that a left Leibniz algebra $\G$ is called simple if, by definition, the only ideals of $\G$ are $\{0\}$, $\Lei(\G)
$ and $\G$ and $\G\bullet\G\not=\Lei(\G)$. It is called  semi-simple if its radical is equal to $\Lei(\G)$ which is equivalent to $\Lei(\G)$ is semi-simple. One can see \cite{aypl} for details.

From what above we can derive the following important results.

\begin{pr}\label{semisimple} Let $(\G,\bullet,\prs)$ be a pseudo-Euclidean  left Leibniz algebra. Then the following assertions hold.
	\begin{enumerate}
		\item If $(\G,\bullet,\prs)$ is $\Li$-quadratic and simple then either $\Lei(\G)=\Lei(\G)^\perp$, $A=\{0\}$ and $\h=\Lie(\G)$ or $\Lei(\G)$ is nondegenerate, $\Lei(\G)^\perp$ is a simple quadratic Lie algebra and $\G=\Lei(\G)\oplus\Lei(\G)^\perp$.
		\item If  $(\G,\bullet,\prs)$ is $\Ri$-quadratic and semi-simple then $(\G,\bullet,\prs)$ is a quadratic Lie algebra.
	\end{enumerate}
	
\end{pr}

\begin{proof} \begin{enumerate}
		\item According to Proposition \ref{leib}, $\Lei(\G)+\Lei(\G)^\perp$ is an ideal and hence $\Lei(\G)+\Lei(\G)^\perp=\Lei(\G)$ or $\Lei(\G)+\Lei(\G)^\perp=\G$. In the first case $\Lei(\G)^\perp\subset\Lei(\G)$ and $\Lei(\G)^\perp$ becomes an ideal and  $\Lei(\G)^\perp=\Lei(\G)$. Thus  $A=\{0\}$ and $\h=\Lie(\G)$. In the second case $\Lei(\G)$ is nondegenerate,  $\G=\Lei(\G)\oplus\Lei(\G)^\perp$ and $\Lei(\G)^\perp$ is a simple quadratic Lie algebra.
		\item If $\G$ is semi-simple then $\G\bullet\G=\G$ but, according to Proposition \ref{leib}, $\G\bullet\G\subset\Lei(\G)^\perp$ and hence $\Lei(\G)=\{0\}$ which completes the proof.\qedhere
	\end{enumerate}
	
\end{proof}

\section{  Description of  $\Li$-quadratic left Leibniz algebras with degenerate Leibniz ideal  }\label{section5}
  
  In this section, we utilize the core of a $\Li$-quadratic left Leibniz algebra, as defined in Definition \ref{def}, to reconstruct the original algebra using a process that resembles a combination of double extension and $T^*$-extension. This reconstruction method enables us to derive several significant results.
The double extension  is a well-known process which appeared first in \cite{Medina} in the study of quadratic Lie algebras while the $T^*$-extension was introduced  in \cite{bor} in the study of metrised nonassociative algebras.

For an  algebra $(A, \bullet)$,  
 $\mathrm{Der}(A)$ is the Lie algebra of derivations of $A$, and if $A$ carries a pseudo-Euclidean metric $\prs_A$, $\mathrm{so}(A)$ is the Lie algebra of skew-symmetric endomorphisms of $A$ and, for any vector space $\h$, we call natural the metric $\prs_n$ on $\h\oplus A\oplus\h^*$ given by
\begin{equation}\label{natural}
	\langle X+a+\al,X+a+\al\rangle_n=2\prec\al,X\rangle+\langle a,a\rangle_A,\quad X\in\h,a\in A,\al\in\h^*.
\end{equation}

For a Lie algebra $(\G,\br)$, $\ad: \G \too \mathrm{End}(\G)$ is the adjoint representation, $\ad^*: \G \too \mathrm{End}(\G^*)$ the coadjoint representation give by $\prec \ad_X^*\al,Y\succ=-\prec\al,[X,Y]\succ.$ The coadjoint representation extends to $\wedge^2\G^*$ by the formula
\[ \ad_X^*\om(Y,Z)=-\om([X,Y],Z)-\om(Y,[X,Z]),\quad \om\in\wedge^2\G^*,X,Y,Z\in\G. \]
For $\Om:\G\too\wedge^2\G^*$, the differential of $\Om$ is the skew-symmetric bilinear map $\De\Om:\G\times\G\too\wedge^2\G^*$ given by
\begin{equation}\label{delta} \De(\Om)(X,Y)=\ad_X^*\Om(Y)-\ad_Y^*\Om(X)-\Om([X,Y]) 
	\end{equation}
Finally, for any $\om\in\wedge^2\G^*$, we denote by $\om^\flat:\G\too\G^*$ the linear map given by $\prec\om^\flat(X),Y\succ=\om(X,Y)$.

Let us state one of our  main result.
\begin{theo}\label{main3} Let $(\G,\bullet_\G,\prs_\G)$ be a $\Li$-quadratic left Leibniz algebra such that $\Lei(\G)$ is degenerate and let  $(\h,\br_\h)$ and $(A,\bullet_A,\prs_A)$ be its core. Then
	$(\g, \bullet_\G, \prs_\G)$    is isomorphic to $(\h\oplus A \oplus \h^{*}, \bullet, \prs_n)$ where, for any $X,Y\in \h$, $a,b\in A$, $\al\in\h^*$, the non vanishing Leibniz products are given by    
	\begin{equation}\label{dd} \begin{cases}X\bullet Y=[X, Y]_\h+\theta(X,Y)+\Om(X)^\flat(Y),\;
			a\bullet  b=a\bullet_Ab+\mu(a,b),,\\
			X\bullet a=F(X)a+\nu(X,a),\;
			X\bullet \al=(\mathrm{ad}^{\h})_{X}^*\al,\;
			a\bullet X=G(X)a+\rho(a,X),
		\end{cases} 
	\end{equation}
and
	 $F:\h\too \mathrm{so}(A)\cap\mathrm{Der}(A)$,  $G:\h\too \mathrm{End}(A)$, $\Om:\h\too \wedge^2\h^*$  are linear maps, 
		 $\theta,\om:\h\times\h\too A$  are bilinear maps such that $\om$ is skew-symmetric and $\mu,\rho,\nu$ are defined by the following relations:
			\begin{equation} \label{relations} \prec\rho(a,X),Y\succ=-\langle\om(X,Y),a\rangle_A,\;
			 \prec\mu(a,b),X\succ=-\langle G(X)a,b\rangle_A,\; \prec\nu(X,a),Y\succ=
			-\langle \theta(X,Y),a\rangle_A.
		  \end{equation} 
	Moreover,  $(F,G,\om, \theta, \Om)$ satisfy the following equations, for any $X,Y,Z,T\in\h$ and $a\in A$,
	\begin{eqnarray}\label{eqmain3} \begin{cases}
			[\Li_a^A,F(X)]=\Li_{G(X)a}^A,\; [\Li_a^A,G(X)]=\Ri^A_{G(X)a}\\
			\Li_{\om(X,Y)}^A=G^*(X)G(Y)-G^*(Y)G(X),\\
			\Ri_{\theta(X,Y)}=[F(X),G(Y)]-G([X,Y]_\h)=F(X)G(Y)+G(Y)G(X)-G([X,Y]_{\h}),\\
			\Li^{A}_{\theta(X,Y)}=[F(X),F(Y)]-F([X, Y]_\h),\\
			\theta(X,[Y, Z]_\h)-\theta([X, Y]_\h,Z)	-
			\theta(Y,[X, Z]_\h)+F(X)\theta(Y,Z)-F(Y)\theta(X,Z)-G(Z)\theta(X,Y)=0,\\
			G(Z)^*\theta(X,Y)-G(Y)^*\theta(X,Z)-F(X)\om(Y,Z)+\om(Y,[X,Z]_\h)+\om([X,Y]_\h,Z)=0,\\
			 (G(X)^*+F(X))\om(Y,Z)=0,\\
			\De(\Om)(X,Y)(Z,T)=
			\langle \theta(X,T),\theta(Y,Z)\rangle_A-\langle \theta(Y,T),\theta(X,Z)\rangle_A
			-\langle \om(Z,T),\theta(X,Y)\rangle_A,
				\end{cases}\end{eqnarray}where $\De$ is given by \eqref{delta}.

Conversely, if $(A,\bullet_A,\prs_A)$ is a $\Li$-quadratic left Leibniz algebra, $(\h,\br_\h)$ is a Lie algebra and $(F,G,\om, \theta, \Om)$ a list as above satisfying \eqref{relations} and \eqref{eqmain3} then $(\h\oplus A\oplus \h^*,\bullet,\prs_n)$ where $\bullet$ is given by \eqref{dd} is a $\Li$-quadratic left Leibniz algebra. We call it the $\Li$-quadratic left algebra obtained form $(A,\bullet_A,\prs_A)$ and $(\h,\br_\h)$ by means of $(F,G,\om,\theta,\Om)$.

\end{theo}

\begin{proof}We prove the theorem in two steps.
	 
		{\bf First step.}	 In this step we build all the data given in the theorem.
		
		Let $(\G,\bullet,\prs)$ be a degenerate $\Ll$-invariant Leibniz algebra. Let $(\h,\br_\h)$ 	and $(A,\bullet_A,\prs_A)$ be its core defined in Definition \ref{def}. We denote by $I=\Lei(\G)\cap \Lei(\G)^\perp$.  Choose  a vector subspace $B$ such that $I^\perp=B\oplus I$.
		$B^\perp$ contains $I$ and $\dim B^\perp=2\dim I$. So there exists a totally isotropic vector subspace $\h_0$ such that $B^\perp=I\oplus \h_0$.  Now, the metric on $B^\perp$ permits the identification of $I$ to $\h_0^*$ and hence $(B^\perp,\prs)$ is isometric $\h_0\oplus\h_0^*$ endowed with its canonical neutral metric.  We obtain that
		$ \G=\h_0\oplus B\oplus \h_0^*$ and the metric is $\prs_n$.
		
		Now, $I$ is a totally isotropic  ideal such that for any $u\in I$, $\mathrm{L}_u=0$, $I^\perp=\Lei(\G)+\Lei(\G)^\perp$ and, by virtue of Proposition \ref{leib}, $I^\perp$ is an ideal and $I^\perp\bullet I=0$. So the Leibniz product can be written, for any $X,Y\in \h$, $a,b\in B$,  $\al\in\h^*$,    
		\begin{equation*} \begin{cases}X\bullet Y=[X, Y]_{\h_0}+\theta(X,Y)+\theta_2(X,Y),\;
				a\bullet  b=a\bullet_Bb+\mu(a,b),\;\al\bullet X=\al\bullet a=0,\\
				X\bullet a=F(X)a+\nu(X,a),\;
				X\bullet \al=H(X)\al,\;
				a\bullet X=G(X)a+\rho(a,X),\; a\bullet\al=0,
			\end{cases} 
		\end{equation*}where $\theta(X,Y)\in B$, $\theta_2(X,Y)\in\h_0^*$, $a\bullet_B b\in B$, $\mu(X,a),\nu(X,a),\rho(a,X)\in\h_0^*$, $F(X),G(X)\in\mathrm{End}(A)$ and $H(X)\in\mathrm{End}(\h_0^*)$. 
	
	    Let $\pi_\h:\G\too\h=\G/I^\perp$ and $\pi_A:I^\perp\too A=I^\perp/I$. It is clear that $\pi_\h$ identifies $(\h_0,\br_{\h_0})$ with $(\h,\br_\h)$ and $\pi_A$ identifies $(B,\bullet_B,\prs_B)$ with $(A,\bullet_A,\prs_A)$. This shows that $\br_{\h_0}$ is a Lie bracket and $\bullet_B$ is a left Leibniz bracket. We identify $\h_0$ to $\h$ and $B$ to $A$. Furthermore, the relation $\langle X\bullet \al,Y\rangle=-\langle \al,X\bullet Y\rangle$ gives that $H(X)=(\ad_X^\h)^*$.
		For any $X,Y,Z\in\h$
		\[ \langle X\bullet Y,Z\rangle_n=\prec\theta_2(X,Y),Z\succ=-
		\langle X\bullet Z,Y\rangle_n=-\prec\theta_2(X,Z),Y\succ. \]
		Define $\Om:\h\too\wedge^2\h^*$ by the relation $\Om(X)(Y,Z)=\prec\theta_2(X,Y),Z\rangle$ and hence $\theta_2(X,Y)=\Om(X)^\flat(Y)$.
		We introduce $\om\in\wedge^2\h^*$ by the relation $\prec\rho(a,X),Y\succ=-\langle\om(X,Y),a\rangle_A$ and the other relation in \eqref{relations} are a direct consequence of the fact that the metric is $\Li$-quadratic.

		{\bf Second step.} In this step we show that the product given by \eqref{dd} is left Leibniz and the metric is $\Li$-quadratic if and only if \eqref{eqmain3} holds and $F$ takes its values in $\mathrm{Der}(A)\cap\mathrm{so}(A)$.  This is a long and tedious computation we defer to the appendix.\qedhere

\end{proof}

\begin{co}\label{zero}Let $(\G,\bullet_\G,\prs_\G)$ be a $\Li$-quadratic left Leibniz algebra, $(\h,\br_\h)$ and $(A,\bullet_A,\prs_A)$ its core. If  $\Lei(\G)^\perp=\Lei(\G)$ then $A=\{0\}$, $\h=\Lie(\G)$ and $(\G,\bullet_\G,\prs_\G)$ is isomorphic to $(\h\oplus\h^*,\bullet,\prs_n)$ where, for any $X,Y\in\h$, $\al,\be\in\h^*$
	\begin{equation}\label{comoh} (X+\al)\bullet (Y+\be)=[X, Y]_\h+(\ad^{\h})_X^*\be+\Om(X)^\flat(Y),\;
	\end{equation}and $\Om:\h\too\wedge^2\h^*$ is a 1-cocycle, i.e., $\De(\Om)=0$.
	Moreover, two such structures of $\Li$-quadratic left Leibniz algebras on $\h\oplus\h^*$ associated to two 1-cocycles $\Om_1$ and $\Om_2$ are isomorphic  if $\Om_1$ and $\Om_2$ define the same cohomolgy class in $H^2(\h,\wedge^2\h^*)$.
	
\end{co}
\begin{proof} The first part is a consequence of Theorem \ref{main3}. Let $(\h,\br)$ be a Lie algebra and $\Om_1$ and $\Om_2$ two 1-cocycles on $\h$ with values in $\wedge^2\h^*$. Denote by $\bullet_i$ the left Leibniz product on $\h\oplus\h^*$ associated to $\Om_i$ and given by \eqref{comoh}. If $\Om_1$ and $\Om_2$ define the same cohomolgy class in $H^2(\h,\wedge^2\h^*)$ then there exists $\om\in\wedge^2\h^*$ such that for any $X\in\h$, $\Om_2(X)=\Om_1(X)+\ad_X^*\om$. On can check easily that the isomorphism $\phi:\h\oplus\h^*\too\h\oplus\h^*$, $X+\al\mapsto X+\al-\om^\flat(X)$ is an isomorphism of $\Li$-quadratic left Leibniz algebras between  $(\h\oplus  \h^*,\bullet_1,\prs_n)$ and $(\h\oplus  \h^*,\bullet_2,\prs_n)$.
	\end{proof}

As a consequence we can give a complete description of $\Li$-quadratic simple left Leibniz algebras.

\begin{co} $(\G,\bullet_\G,\prs_\G)$ be a $\Li$-quadratic simple left Leibniz algebra. Then $(\G,\bullet_\G,\prs_\G)$ is isomorphic to $(\h\oplus\h^*,\bullet,\prs_n)$ where $\h=\Lie(\G)$ and, for any $X,Y\in\h$, $\al,\be\in\h^*$
	\[ (X+\al)\bullet (Y+\be)=[X, Y]_\h+(\ad_X^{\h})^*\be.
	\]
	
\end{co}
\begin{proof} According to Proposition \ref{semisimple}, $A=0$ and $\h=\Lei(\G)$ which is a simple Lie algebra. We can apply Corollary \ref{zero} and in addition, since $\h$ is simple, $\Om$ is a co-boundary which completes the proof.
	\end{proof}

Let us recall a property of quadratic Lie algebras we will use in the following Theorem.

Let $(\G,\br,\prs)$ be a quadratic Lie algebra and $\om\in\wedge^2\G^*$. Then there exists a skew-symmetric endomorphism $D:\G\too\G$ such that, for any $x,y\in\G$,
\[ \om(x,y)=\langle Dx,y\rangle. \]
It is a well-known fact which is easy to establish that $\om$ is a 2-cocycle of $\G$ if and only if $D$ is a derivation. Hence $H^2(\G)=(\mathrm{so}(\G)\cap\mathrm{Der}(\G))/\mathrm{Inner}(\G)$ where $\mathrm{Inner}(\G)=\{\ad_u,u\in\G\}$.

\begin{theo}\label{semil}Let $(\h\oplus A\oplus\h^*,\bullet,\prs_n)$ be the $\Li$-quadratic left algebra obtained form $(A,\bullet_A,\prs_A)$ and $(\h,\br_\h)$ by means of $(F,G,\om,\theta,\Om)$.
	
  If $A$ is a Lie algebra, $Z(A)=0$ and $H^2(A)=0$ then there exists a 1-cocycle $\Om_0$ of $\h$ with values in $\wedge^2\h^*$ and  an endomorphism $U:\h\too A$ such that $(\h\oplus A\oplus\h^*,\bullet,\prs_n)$ is isomorphic to $(\h\oplus A\oplus \h^*,\circ,\prs)$ where, for any $X,Y\in\h$, $a,b\in A$, $\al,\be\in\h^*$, the non vanishing  Leibniz products and the metric are given by
	\begin{align*}& X\circ Y=[X,Y]_\h+\Om_0(X)^\flat(Y),\; a\circ b=[a,b]_A+U^*([a,b]_A),\; 
		 X\circ \al=(\ad_X^\h)^*\al,\\
	 &\langle X+a+\al,Y+b+\be\rangle=\langle X+a+\al,Y+b+\be\rangle_n+\langle U(X),U(Y)\rangle_A-\langle a,U(Y)\rangle_A-\langle b,U(X)\rangle_A,\\
&\prec U^*(a),X\succ=\langle a,U(X)\rangle_A. \end{align*}
	 
\end{theo}

\begin{proof}  The structure of $\Li$-quadratic left Leibniz algebra of $\h\oplus A\oplus\h^*$ is given in Theorem \ref{main3} with the conditions satisfied by $(F,G,\om,\theta,\Om)$. Suppose that $A$ is a Lie algebra,  $Z(A)=0$ and $H^2(A)=0$. 
	Then $[A,A]_A=Z(A)^\perp=A$ and any skew-symmetric derivation is inner.  From  the first equation in \eqref{eqmain3},  the fact that $F$ is a derivation and $Z(A)=0$ we deduce that $G=-F$.
	 So exists a unique endomorphism $U:\h\too A$ such that
	$F(X)=-G(X)=\ad_{U(X)}$ and from the third equation in  \eqref{eqmain3}, we must have $$
	\theta(X,Y)=[U(X),U(Y)]_A-U([X,Y]_\h).$$
	We consider now the isomorphism $\phi:\h\oplus A\oplus\h^*\too \h\oplus A\oplus\h^*$ given by $\phi(X+a+\al)=X+a-U(X)+\al$. Under this isomorphism, $(\h\oplus A\oplus \h^*,\bullet,\prs_n)$ is isomorphic to $(\h\oplus A\oplus \h^*,\circ,\prs)$ where, for any $u,v\in \h\oplus A\oplus \h^*$,
	\[ u\circ v=\phi^{-1}(\phi(u)\bullet \phi(v))\esp \langle u,v\rangle=\langle \phi(u),\phi(v)\rangle_n. \]So
	\[ \langle X+a+\al,Y+b+\be\rangle=\langle X+a+\al,X+b+\al\rangle_n+\langle U(X),U(Y)\rangle_A-\langle a,U(Y)\rangle_A-\langle b,U(X)\rangle_A. \]
	For any $X,Y\in\h$,
	\begin{align*}
		\phi(X\circ Y)&= (X-U(X))\bullet (Y-U(Y))\\
		&=[X,Y]_\h+\theta(X,Y)+\Om(X)^\flat(Y)-[U(X),U(Y)]_A-\nu(X,U(Y))+[U(Y),U(X)]_A\\&-\rho(U(X),Y)+[U(X),U(Y)]_A+\mu(U(X),U(Y))\\
		&=[X,Y]_\h-U([X,Y]_\h)+\Om(X)^\flat(Y)-\rho(U(X),Y)+\mu(U(X),U(Y)).
	\end{align*}So there exist $\theta_2$ such that
\[ X\circ Y=[X,Y]_\h+\theta_2(X,Y). \]
From the relation
\[ \langle X\circ Y,Z\rangle=\prec \theta_2(X,Y),Z\succ=-\langle X\circ Z,Y\rangle=-\prec \theta_2(X,Z),Y\succ \]we can define
	 $\Om_0:\h\too\wedge^2\h^*$ by the relation $\Om_0(X)(Y,Z)=\prec \theta_2(X,Y),Z\succ$ and write
	\[ X\circ Y=[X,Y]_\h+ \Om_0(X)^\flat(Y).\]
	In the same way, we can show that there exists, $\nu^0,\rho^0$ such that for any $X\in\h$, $a,b\in A$, $\al\in\h^*$,
	\[ a\circ b=[a,b]_a+ \mu(a,b),\; X\circ a=\nu^0(X,a),\; X\circ \al=(\ad_X^\h)^*\al\esp a\circ X=\rho^0(a,X). \]
	From the relations
	\[ \langle a\circ b,X\rangle=-\langle b,a\circ X\rangle\esp\langle X\circ a,Y\rangle=-\langle a, X\circ Y\rangle \]we get
	\[ \nu^0=0\esp \prec\mu(a,b),X\succ=\langle[a,b]_A,U(X)\rangle_A. \]
	The bracket $\circ$ has the same form as \eqref{dd} with $F=G=0$, $\theta=0$ and $\nu=0$ and \eqref{eqmain3} are the conditions for this bracket to be left Leibniz. So we get from the
	equation (2) in \eqref{eqmain3}  that $\om_0=0$ and the last equation in \eqref{eqmain3} gives that $\Om_0$ is a 1-cocycle. This completes the proof.
\end{proof}

\begin{remark} The class of quadratic Lie algebras $\G$ satisfying $Z(\G)=\{0\}$ and $H^2(\G)=\{0\}$ is large. It contains the class of semi-simple Lie algebras. It contains also the  cotangents of simple Lie algebras. Indeed, if $(\G,\br_\G)$ is a Lie algebra its cotangent $T^*\G$ is the Lie algebra $\G\oplus\G^*$ with the bracket
	\[ [X+\al,Y+\be]=[X,Y]_\G+\ad_X^*\al-\ad_Y^*\be,\quad X,Y\in \G,\al,\be\in\G^*. \]
	The natural metric on $T^*\G$ is quadratic. Moreover, if $\G$ is simple then $Z(T^*\G)=\{0\}$ and $H^2(T^*\G)=\{0\}$ (see \cite[Theorem 2.2]{bilham}).
	
\end{remark}

\begin{co}\label{diml} Let $(\G,\bullet,\prs)$ be a $\Li$-quadratic left Leibniz algebra which is not a Lie algebra. Then $\dim\Lei(\G)\geq2$.

\end{co}

\begin{proof} Suppose that $\dim\Lei(\G)=1$. If $\Lei(\G)$ is nondegenerate then $\G=\Lei(\G)\oplus\Lei(\G)^\perp$ and the action of $\Lei(\G)^\perp$ on $\Lei(\G)$ is made by skew-symmetric endomorphisms so it is trivial and hence $\G$ is a Lie algebra.

	Suppose tha $\Lei(\G)$ is degenerate then $\Lei(\G)\subset\Lei(\G)^\perp$ and hence $\dim\h=1$. So $(\G,\bullet,\prs)$ is isomorphic to the $\Li$-quadratic left Leibniz algebra $(\h\oplus A\oplus \h^*,\bullet,\prs_n)$ described in Theorem \ref{main3} with the data $(F,G,\om,\theta,\Om)$. Since $\dim\h=1$, we have $\Om=0$ and $\om=0$.  
	Moreover, $\h^*=\Lei(\G)$ and hence for any $X,Y\in\h$ and $a\in A$,
	\[ X\bullet Y+Y\bullet X=\theta(X,Y)+\theta(Y,X)\in\h^* \esp X\bullet a+a\bullet X=F(X)a+G(X)a+\nu(X,a)\in\h^* \]this implies that $F=-G$,  $\theta=0$ and $\nu=0$. The product $\bullet$ becomes skew-symmetric and hence $\G$ is a Lie algebra which is a contradiction.
	\end{proof}

We end this section by giving a complete description of  $\Li$-quadratic left Leibniz algebras such that $\dim(\Lei(\G)\cap\Lei(\G)^\perp)=1$. This class contains the class of Lorentzian $\Li$-quadratic left Leibniz algebras with degenerate Leibniz ideal. However,
 let $(\G,\bullet,\prs)$ be a Lorentzian $\Ll$-quadratic Leibniz algebra such that $\Lei(\G)$ is nondegenerate. Then $\G=\Lei(\G)\oplus\Lei(\G)^\perp$ where $\Lei(\G)^\perp$ is either a quadratic Euclidean Lie algebra or a quadratic Lorentzian Lie algebra (see Section \ref{section4}).

\begin{theo}\label{LorentiznL} Let $(\G,\bullet,\prs)$ be a $\Li$-quadratic  left Leibniz algebra such that $\dim(\Lei(\G)\cap\Lei(\G)^\perp)=1$. Then $(\G,\bullet,\prs)$ is isomorphic to $(\R e\oplus A\oplus\R \bar{e},\bullet,\prs_n)$ where $(A,\bullet_A,\prs_A)$ is a $\Li$-quadratic   left Leibniz algebra  and the non vanishing Leibniz products are given by
	\begin{equation*} 
			\bar{e}\bullet \bar{e}=\de,\;
			a\bullet b=a\bullet_A b-\langle Ga,b\rangle_A e,\;
			\bar{e}\bullet a=Fa-\langle\de,a\rangle e,\;
			a\bullet \bar{e}=Ga,
\end{equation*}where $a,b\in A$, $\de\in A$,  $F\in\mathrm{Der}(A)\cap\mathrm{so}(A)$, $G\in\mathrm{End}(A)$ satisfying
\[ G(\de)=0,\;\Li_{\de}^A=0,  \; \mathrm{Im}(F+G)\in Z^l(A),\; [ \Li_a^A,G]=\Ri_{Ga}^A\esp[G,F]= G^2+GF=\Ri_\de. \]

\end{theo}

\begin{proof}  According to Theorem \ref{main3}, $(\G,\bullet,\prs)$ is isomorphic to $\G=\h\oplus A\oplus \h^*$ where $(A,\bullet_A,\prs_A)$ is a $\Li$-quadratic Euclidean left Leibniz algebra,
	$\h=\R\bar{e}$,  $\h^*=\R e$ and the Leibniz product is given by \eqref{dd}. Since $\dim\h=1$, we have $\Om=0$ and $\om=0$. Put $$F=F(\bar{e}),\; G=G(e),\;\theta(\bar{e},\bar{e})=\de\in A,\; \mu(a,b)=-\langle Ga,b\rangle_A e.$$ Then
	$\nu(\bar{e},a)=-\langle\de,a\rangle e$ and we get the desired expression for the Leibniz product. One can check easily that \eqref{eqmain3} reduces to the equations given in the statement of the theorem.
	\end{proof}

\section{Description of $\Ri$-quadratic left Leibniz algebras} \label{section5bis}

In this section, we adopt a similar approach as in the previous section to give a complete description $\Ri$-quadratic left Leibniz algebras and to derive some important results. We use the same notations as the previous section.

The following theorem has many similarities with Theorem \ref{main3}.
\begin{theo}\label{main4} Let $(\G,\bullet_\G,\prs_\G)$ be a $\Ri$-quadratic left Leibniz algebra,  $(\h,\br_\h)$ and $(A,\br_A,\prs_A)$  its core. Then
	$(\g, \bullet_\G, \prs_\G)$    is isomorphic to $(\h\oplus A \oplus \h^{*}, \br, \prs_n)$ where    for  $X,Y\in \h$, $a,b\in A$,  $\al\in\h^*$, the non vanishing products are given by
	\begin{equation}\label{ddr} \begin{cases}X\bullet Y=\theta(X,Y)+\Om(Y)^\flat(X),\;
			a\bullet b=[a,b]_A+\mu(a,b),\\
			X\bullet a=F(X)a+\nu(X,a),\;
			X\bullet \al=(\ad^\h)_X^*\al,\;
			a\bullet X=-F(X)a+\rho(a,X),\; 
		\end{cases} 
	\end{equation} and
	  $F:\h\too \mathrm{so}(A)\cap\mathrm{Der}(A)$,
		$\Om:\h\too\wedge^2\h^*$ are linear maps, $\theta,\om:\h\times\h\too A$ are bilinear skew-symmetric,  and $\mu,\rho,\nu$ are defined by the relations
		\begin{equation} \label{relationr} \prec\nu(X,a),Y\succ=\langle\om(X,Y),a\rangle_A,\;
			\prec\mu(a,b),X\succ=\langle F(X)a,b\rangle_A, \; \esp  \prec\rho(a,X),Y\succ=\langle \theta(X,Y),a\rangle_A.
			 \end{equation}
	Moreover,  $(F,\om, \theta, \Om)$ satisfy the following equations, for any $X,Y,Z,T\in\h$,
		\begin{equation}\label{eqmain4}
		\begin{cases}\ad_{\om(X,Y)}^A=F([X,Y]_\h)-[F(X),F(Y)],\\
			\ad^{A}_{\theta(X,Y)}=[F(X),F(Y)],\\
			F(X)\theta(Y,Z)+F(Y)\theta(Z,X)+F(Z)\theta(X,Y)=0,\\
			F(X)\om(Y,Z)+F(Y)\om(Z,X)-F(Z)\theta(X,Y)
			+\om(X,[Y,Z]_\h)-\om(Y,[X,Z]_\h)=0,\\
			F(X)\theta(Y,Z)-F(Y)\om(Z,X)+F(Z)\theta(X,Y)
			-\theta(Y,[X,Z]_\h)=0,\\
			(\ad^\h_T)^*\Om(Z)(X,Y)=
			\langle \om(Y,T),\theta(X,Z)\rangle_A+
			\langle \om(X,T),\theta(Z,Y)\rangle_A+\langle \theta(Z,T),\theta(X,Y)\rangle_A.
	\end{cases}\end{equation}

Conversely, if $(A,\br_A,\prs_A)$ is a quadratic Lie algebra, $(\h,\br_\h)$ is a Lie algebra and $(F,\om, \theta, \Om)$ a list as above satisfying \eqref{relationr} and \eqref{eqmain4}  then $(\h\oplus A\oplus \h^*,\br,\prs_n)$ is a $\Ri$-quadratic left Leibniz algebra. We call it the $\Ri$-quadratic left algebra obtained form $(A,\br_A,\prs_A)$ and $(\h,\br_\h)$ by means of $(F,\om,\theta,\Om)$.

\end{theo}

\begin{proof} We follow the same steps as in the proof of Theorem~\ref{main3}. 
	
	Let $(\G,\bullet,\prs)$ be a degenerate $\Ri$-invariant Leibniz algebra. Let $(\h,\br_\h)$ 	and $(A,\br_A,\prs_A)$ be its core defined in Definition \ref{def}.

	We denote by $I=\Lei(\G)$.  Choose  a vector subspace $B$ such that $I^\perp=B\oplus I$ and as in the proof of Theorem~\ref{main3} we have 
		$ \G=\h_0\oplus B\oplus \h_0^*$ and the metric is $\prs_n$.
	
	Now, $I$ is a totally isotropic  ideal such that for any $u\in I$, $\mathrm{L}_u=0$ and, by virtue of Proposition \ref{leib}, $I^\perp$ is an ideal,  $I^\perp\bullet I=0$ and $\G\bullet\G\subset I^\perp$. So the Leibniz product can be written, for any $X,Y\in \h$, $a,b\in B$,  $\al\in\h^*$,    
	\begin{equation*} \begin{cases}X\bullet Y=\theta(X,Y)+\theta_2(X,Y),\;
			a\bullet  b=a\bullet_Bb+\mu(a,b),\;\al\bullet X=\al\bullet a=0,\\
			X\bullet a=F(X)a+\nu(X,a),\;
			X\bullet \al=H(X)\al,\;
			a\bullet X=G(X)a+\rho(a,X),\; a\bullet\al=0,
		\end{cases} 
	\end{equation*}where $\theta(X,Y)\in B$, $\theta_2(X,Y)\in\h_0^*$, $a\bullet_B b\in B$, $\mu(X,a),\nu(X,a),\rho(a,X)\in\h_0^*$, $F(X),G(X)\in\mathrm{End}(A)$, $H(X)\in\mathrm{End}(\h_0^*)$. We have $X\bullet a+a\bullet X\in I$ and hence $G(X)=-F(X)$. We have also $X\bullet Y+Y\bullet X\in I$ and hence $\theta$ is skew-symmetric.
	
	Let $\pi_\h:\G\too\h=\G/I^\perp$ and $\pi_A:I^\perp\too A=I^\perp/I$. It is clear that $\pi_\h$ identifies $\h_0$ to $\h$ and $\pi_A$ identifies $(B,\bullet_B,\prs_B)$ to $(A,\br_A,\prs_A)$. This shows that $\bullet_B$ is a Lie bracket. We identify $\h_0$ to $\h$ and $B$ to $A$. Furthermore, the relation $$\prec X\bullet \al,Y\succ=\prec \al,\Li_X^*Y\succ\stackrel{\eqref{sub}}=
	-\prec\al,[X,Y]_\h\succ$$ gives that $H(X)=(\ad_X^\h)^*$.
	For any $X,Y,Z\in\h$
	\[ \langle X\bullet Y,Z\rangle_n=\prec\theta_2(X,Y),Z\succ=-
	\langle Z\bullet Y,X\rangle_n=-\prec\theta_2(Z,Y),X\succ. \]
	Define $\Om:\h\too\wedge^2\h^*$ by the relation $\Om(Y)(X,Z)=\prec\theta_2(X,Y),Z\rangle$ and hence $\theta_2(X,Y)=\Om(Y)^\flat(X)$.
	We introduce $\om\in\wedge^2\h^*$ by the relation $\prec\nu(X,a),Y\succ=\langle\om(X,Y),a\rangle_A$ and the other relations in \eqref{relationr} are a direct consequence of the fact that the metric is $\Ri$-quadratic.

	The second step of the proof is similar to the second step of the proof of Theorem \ref{main3} and we defer it also to the appendix.
\end{proof}

\begin{co} Let $(\G,\bullet_\G,\prs_\G)$ be a $\Ri$-quadratic left Leibniz algebra such, $(\h,\br_\h)$ and $(A,\br_A,\prs_A)$ its core. If  $\Lei(\G)^\perp=\Lei(\G)$ then $A=\{0\}$ and $(\G,\bullet_\G,\prs_\G)$ is isomorphic to $(\h\oplus\h^*,\bullet,\prs_n)$ where, for any $X,Y\in\h$, $\al,\be\in\h^*$   
	\[ (X+\al)\bullet (Y+\be)=(\ad_X^\h)^*\be+\Om(Y)^\flat(X),\]and $\Om:\h\too\wedge^2\h^*$ satisfies $\Om(X)^\flat([\h,\h]_\h)=0$ for any $X\in\h$.
	
\end{co}

\begin{theo}\label{semir} Let $(\h\oplus A\oplus\h^*,\br,\prs_n)$ be the $\Ri$-quadratic left algebra obtained form $(A,\br_A,\prs_A)$ and $(\h,\br_\h)$ by means of $(F,\om,\theta,\Om)$.
	
	If $Z(A)=0$ and $H^2(A)=0$ then there exists an endomorphism $U:\h\too A$, $\Om_0:\h\too\wedge^2\h^*$ and $\om_0\in\wedge^2\h^*$ such that $(\h\oplus A\oplus\h^*,\br,\prs_n)$ is isomorphic to to $(\h\oplus A\oplus \h^*,\bullet,\prs)$ where, for any $X,Y\in\h$, $a,b\in A$, $\al,\be\in\h^*$, the non vanishing Leibniz products   
	\begin{align*}
	&	X\bullet Y=\Om(Y)^\flat(X),\; a\bullet b=[a,b]_A+U^*([a,b]_A), X\bullet \al=(\ad_X^\h)^*\al,\\
	&\langle X+a+\al,Y+b+\be\rangle_0=\langle X+a+\al,Y+b+\be\rangle_n+\langle U(X),U(Y)\rangle_A-\langle a,U(Y)\rangle_A-\langle b,U(X)\rangle_A,
	\end{align*}
	where $\Om_0(X)^\flat([\h,\h]_\h)=0$ for any $X\in\h$.

\end{theo}

\begin{proof} Using the same arguments as those used in th the proof  of Theorem \ref{semil}, we can show that $(\h\oplus A\oplus\h^*,\br,\prs_n)$ is isomorphic to $(\h\oplus A\oplus\h^*,\br,\prs)$ where, for any $X,Y\in\h$, $a,b\in A$, $\al\in\h^*$,
	\[X\bullet Y_0=\theta_2(X,Y),\; a\bullet b=[a,b]_a+ \mu(a,b),\; X\bullet a=\nu(X,a),\; X\bullet \al=(\ad_X^\h)^*\al\esp a\bullet X=\rho(a,X). \]The metric $\prs$ is given by
	\[ \langle X+a+\al,X+a+\al\rangle=2\prec \al,X\succ+\langle a,a\rangle+\langle U(X),U(X)\rangle_A-2\langle a,U(X)\rangle_A. \]
	From the relations
	\[ \langle a\bullet b,X\rangle=-\langle a,X\bullet b\rangle\esp\langle a\bullet X, Y\rangle=-\langle a, Y\bullet X\rangle \]we get
	\[ \rho=0\esp \prec\mu(a,b),X\succ=\langle[a,b]_A,U(X)\rangle_A. \]
	Now from the  first equation  in \eqref{eqmain4}, we get $\nu=0$ and the last equation gives the condition on $\Om_0$.
\end{proof}

\begin{co}\label{dimr} Let $(\G,\br,\prs)$ be a $\Ri$-quadratic left Leibniz which is not a Lie algebra. Then $\dim\Lei(\G)\geq2$.
	
	In particular,  a $\Ri$-quadratic left Leibniz algebra with the metric Lorentzian is a Lie algebra.

\end{co}

\begin{proof} Suppose that $\dim\Lei(\G)=1$. Then $\dim\h=1$ and by using Theorem \ref{main4}, we get that $\G$ is isomorphic to $\h\oplus A\oplus\h^*$ with $\theta=\om=\Om=0$ and the Leibniz product \eqref{ddr} becomes a Lie bracket which is a contradiction. 
	
If $\G$ is a $\Ri$-quadratic left Leibniz algebra with the metric Lorentzian then, since the metric is Lorentzian and $\Lei(\G)\subset\Lei(\G)^\perp$  then $\dim\Lei(\G)\leq 1$ and hence $\dim\Lei(\G)=0$.	
	\end{proof}

By using Theorem \ref{main3} and Theorem \ref{main4}, we give now a complete description of $\Li$-quadratic symmetric Leibniz algebras.

\begin{theo}\label{main} Let $(\G,\bullet_\G,\prs_\G)$ be a $\Li$-quadratic symmetric Leibniz algebra with  $(\h,\br_\h)$ and $(A,\bullet_A,\prs_A)$  its core. Then $\Lei(\G)\subset\Lei(\G)^\perp$, $(A,\bullet_A,\prs_A)$ is a quadratic Lie algebra, $(\h,\br_\h)$ is abelian and
	$(\g, \bullet_\G, \prs_\G)$    is isomorphic to $(\h\oplus A \oplus \h^{*}, \bullet, \prs_n)$  where for any $X,Y\in \h$, $a,b\in A$,  $\al\in\h^*$,  the non vanishing Leibniz products are given by 
	 	\begin{equation} \begin{cases}X\bullet Y=\theta(X,Y)+\Om(X)^\flat(Y),\;
			a\bullet b=[a,b]_A+\mu(a,b),\\
			X\bullet a=F(X)a+\nu(X,a),\;\;
			a\bullet X=-F(X)a+\rho(a,X),
		\end{cases} 
	\end{equation}
	and $F:\h\too \mathrm{so}(A)\cap\mathrm{Der}(A)$, $\Om:\h\too\wedge^2\h^*$ are linear,  
		  $\theta,\om:\h\times\h\too A$     are bilinear  skew-symmetric  and $\mu,\rho,\nu$ are defined by the following relations:
		\[ \prec\rho(a,X),Y\succ=-\langle\om(X,Y),a\rangle_A,\;
			\prec\mu(a,b),X\succ=\langle F(X)a,b\rangle_A\esp \prec\nu(X,a),Y\succ=
			-\langle \theta(X,Y),a\rangle_A.
	 \] 
	
	Moreover,  $(F,\om, \theta, \Om)$ satisfy the following equations:
		\begin{equation}\label{eqmain5}
		\begin{cases}
			\ad^{A}_{\theta(X,Y)}=-\ad_{\om(X,Y)}^A=[F(X),F(Y)],\\
			F(X)\theta(Y,Z)+F(Y)\theta(Z,X)+F(Z)\theta(X,Y)=0,\\
			F(X)\left[\theta(Y,Z)+\om(Y,Z)\right]=0,\\
			\langle \theta(Y,T),\theta(Z,X)\rangle_A+
			\langle \theta(X,T),\theta(Y,Z)\rangle_A+\langle \theta(Z,T),\theta(X,Y)\rangle_A=0,\\
			\langle \theta(X,T)+\om(X,T),\theta(Y,Z)\rangle_A=0,
	\end{cases}\end{equation}for any $X,Y,Z,T\in\h$.

\end{theo}
\begin{proof}
	 According to Theorem \ref{main3},  $(\G,\bullet,\prs)$ is isomorphic to $(\h\oplus A\oplus\h^*)$ with the non vanishing Leibniz products are given, for any  $X,Y\in \h$, $a,b\in A$,  $\al\in\h^*$, by
	\begin{equation*} \begin{cases}X\bullet Y=[X, Y]_\h+\theta(X,Y)+\Om(X)^\flat(Y),\;
			a\bullet b=a\bullet_A b+\mu(a,b),\\
			X\bullet a=F(X)a+\nu(X,a),\;
			X\bullet \al=(\mathrm{ad}^{\h})_{X}^*\al,\;
			a\bullet X=G(X)a+\rho(a,X),
		\end{cases} 
	\end{equation*}
	 and $(F,G,\om, \theta, \Om)$ satisfy \eqref{relations} and \eqref{eqmain3}.
	
	$\G$ is also a right Leibniz algebra and hence the product $\circ$ given by $X\circ Y=Y\bullet X$ is a left Leibniz product and $(\G,\circ,\prs)$ is $\Ri$-quadratic. So $\Lei(\G)\subset \Lei(\G)^\perp$ and hence the core of $(\G,\circ,\prs)$ is $(A,-\bullet_A,\prs_A)$ and $(\h,\br_\circ)$. Thus implies that $\bullet_A$ is Lie bracket, let denote it by $\br_A$. Then, according to Theorem \ref{main4}, $(\G,\circ,\prs)$ is isomorphic to 
	$(\h\oplus A\oplus \h^*)$ with the non vanishing  products
	$$\begin{cases}X\circ Y=\theta^\circ(X,Y)+\Om^\circ(Y)^\flat(X),\;
		a\circ b=-[a,b]_A+\mu^\circ(a,b),\\
		X\circ a=F^\circ(X)a+\nu^\circ(X,a),\;
		X\circ \al=(\mathrm{ad}^{\circ})_{X}^*\al,\;
		a\circ X=-F^\circ(X)a+\rho^\circ(a,X),
	\end{cases} $$ and $(F^\circ,\om^\circ,\theta^\circ,\Om^\circ)$ satisfy \eqref{eqmain4} and \eqref{relationr}. Moreover, $\om^\circ$ and $\theta^\circ$ are skew-symmetric. 
	
	We deduce that $\br_\h=\br_\circ=0$ and
	\[ \theta^\circ=-\theta,\;\Om^\circ=\Om,F^\circ=-F=G\esp \om^0=-\om.\]
	
	If we replace in \eqref{main3} and \eqref{main4} we get the two systems of equations
	
	\begin{equation*} \begin{cases}
			\ad_{\om(X,Y)}^A=-[F(X),F(Y)],\\
			\ad^{A}_{\theta(X,Y)}=[F(X),F(Y)],\\
			F(X)\theta(Y,Z)+F(Y)\theta(Z,X)+F(Z)\theta(X,Y)=0,\\
			F(Z)\theta(X,Y)+F(Y)\theta(Z,X)-F(X)\om(Y,Z)=0,\\
			\langle \theta(X,T),\theta(Y,Z)\rangle_A-\langle \theta(Y,\star),\theta(X,Z)\rangle_A
			-\langle \om(Z,T),\theta(X,Y)\rangle_A=0.
	\end{cases}\end{equation*}
	
		\begin{equation*}
		\begin{cases}\ad_{\om(X,Y)}^A=-[F(X),F(Y)],\\
			\ad^{A}_{\theta(X,Y)}=[F(X),F(Y)],\\
			F(X)\theta(Y,Z)+F(Y)\theta(Z,X)+F(Z)\theta(X,Y)=0,\\
			F(X)\om(Y,Z)+F(Y)\om(Z,X)-F(Z)\theta(X,Y)
			=0,\\
			F(X)\theta(Y,Z)-F(Y)\om(Z,X)+F(Z)\theta(X,Y)
			=0,\\
			\langle \om(Y,T),\theta(X,Z)\rangle_A-
			\langle \om(X,T),\theta(Y,Z)\rangle_A+\langle \theta_1(Z,T),\theta(X,Y)\rangle_A=0.
	\end{cases}\end{equation*}These two systems are equivalent to

\begin{equation*}
	\begin{cases}\ad_{\om(X,Y)}^A=-[F(X),F(Y)],\\
		\ad^{A}_{\theta(X,Y)}=[F(X),F(Y)],\\
		F(X)\theta(Y,Z)+F(Y)\theta(Z,X)+F(Z)\theta(X,Y)=0,\\
		F(X)\om(Y,Z)+F(Y)\om(Z,X)-F(Z)\theta(X,Y)
		=0,\\
		F(X)\theta(Y,Z)-F(Y)\om(Z,X)+F(Z)\theta(X,Y)
		=0,\\
		\langle \om(Y,T),\theta(X,Z)\rangle_A-
		\langle \om(X,T),\theta(Y,Z)\rangle_A+\langle \theta(Z,T),\theta(X,Y)\rangle_A=0,\\
		\langle \om(Y,T),\theta(X,Z)\rangle_A-
		\langle \om(X,T),\theta(Y,Z)\rangle_A+\langle \theta(Z,T),\theta(X,Y)\rangle_A=0.
\end{cases}\end{equation*}
The equations (3), (4) and (5) in the above system are equivalent to
\[ 
	F(X)\theta(Y,Z)+F(Y)\theta(Z,X)+F(Z)\theta(X,Y)=0
	\esp
	F(X)\left[\theta(Y,Z)+\om(Y,Z)\right]=0,
 \]
 which completes the proof.
	\end{proof}

\begin{co} Any Lorentzian $\Li$-quadratic symmetric  Leibniz algebra is a quadratic Lie algebra.
	
\end{co}

\begin{remark} In Theorem \ref{main}, if we take $\theta=\om=0$, we can build for free many examples of $\Li$-quadratic symmetric Leibniz algebras. For instance let $\h$ be a vector space with a basis $(e_1,\ldots,e_r)$, $(A,\prs_A)$ a 2-dimensional abelian Euclidean Lie algebra with a basis $(f_1,f_2)$. For $i=1,\ldots,r$, take $F(e_i)=\left(\begin{matrix}0&\la_i\\-\la_i&0 \end{matrix}
			\right)$. We take $\om=\theta=0$ and $\Om(e_i)=\left(a_{kl}^i \right)_{kl}$ a skew-symmetric matrix. So on $\h\oplus A\oplus\h^*$ we get a structure of symmetric Leibniz algebra $\bullet$ and a $\Li$-quadratic metric $\prs_n$ by
	\[ e_i\bullet e_j=\sum_{k=1}^ra_{jk}^ie_k^*,\; e_i\bullet f_1=-\la_i f_2,\;e_i\bullet f_2=\la_if_1,f_1\bullet f_2=-f_2\bullet f_1=-\sum_{i=1}^r\la_ie_i^*. \]

\end{remark}

\section{Appendix}\label{apend}

In this appendix we give the conditions for which the pseudo-Euclidean left Leibniz algebras $(\h\oplus A\oplus\h^*,\bullet,\prs_n)$ appearing, respectively, in Theorem \ref{main3} and Theorem \ref{main4} is a $\Li$-quadratic (resp. $\Ri$-quadratic) left Leibniz algebra in order to complete the proof of these theorem. 

The Leibniz products \eqref{dd} and \eqref{ddr} can be written
\begin{equation*} \begin{cases}X\bullet Y=\e[X, Y]_\h+\theta(X,Y)+\theta_2(X,Y),\;
		a\bullet b=a\bullet_A b+\mu(a,b),\\
		X\bullet a=F(X)a+\nu(X,a),\;
		X\bullet \al=(\mathrm{ad}^{\h})_{X}^*\al,\;
		a\bullet X=G(X)a+\rho(a,X),\quad \e\in\{0,1\},
	\end{cases} 
\end{equation*}where \begin{enumerate}
\item $\e=1$, $F(X)$ skew-symmetric,  $\theta_2(X,Y)=\Om(X)^\flat(Y)$, $\om$ skew-symmetric, and
\[ \prec\rho(a,X),Y\succ=-\langle\om(X,Y),a\rangle_A,\;
\prec\mu(a,b),X\succ=-\langle G(X)a,b\rangle_A,\; \prec\nu(X,a),Y\succ=
-\langle \theta(X,Y),a\rangle_A \]when the metric is $\Li$-quadratic,
\item $\e=0$, $G=-F$, $F(X)$ skew-symmetric, $\theta$ and $\om$ are skew-symmetric, $\theta_2(X,Y)=\Om(Y)^\flat(X)$ and
\[ \prec\nu(X,a),Y\succ=\langle\om(X,Y),a\rangle_A,\;
\prec\mu(a,b),X\succ=\langle F(X)a,b\rangle_A,\; \prec\rho(a,X),Y\succ=
\langle \theta(X,Y),a\rangle_A \]when the metric is $\Ri$-quadratic,

\end{enumerate}

 This product defines a left Leibniz structure if and only if, for any $u,v,w\in\G$,
\[ Q(u,v,w):=u\bullet(v\bullet w)-(u\bullet v)\bullet w-v\bullet (u\bullet w)=0. \]
Since, for any $\al\in\h^*$, $\Li_\al=0$ and $\h^*$ is a left ideal, we have $Q(\al,.,.)=Q(.,\al,.)=0$. So the vanishing of $Q$ is equivalent to
This is equivalent to
$$\begin{cases}
	Q(X,Y,Z)=Q(X,Y, \al)=  Q(X,Y,a)= Q(X,a,Y)= Q(X,a,\al)=  Q(X,a,b)=0,\\
	Q(a,b,c)= Q(a,b, X)= Q(a,X,b)=Q(a,b, \al)=
	Q(a,X,Y)= Q(a,X, \al)=0,
\end{cases}$$ for any $X,Y,Z\in \h$, $a,b,c\in A$, and $\al,\be \in\h^{*}$. 

On the other hand, when the metric is $\Li$-quadratic, we have, for any $r,u,v,w\in\G$,
\[ \langle Q(u,v,w),r\rangle_n=-\langle Q(u,v,r),w\rangle_n. \]
We will use this relation to ovoid redundant equations.

\begin{enumerate}  
	\item We have 
	\begin{align*}
		X\bullet(Y\bullet Z)&=X\bullet(\e[Y, Z]_\h+\theta(Y,Z)+\theta_2(Y,Z))\\
		&=\e[X,[Y,Z]_\h]_\h+\e\theta(X,[Y,Z]_\h)
		+\e\theta_2(X,[Y,Z]_\h)+F(X)\theta(Y,Z)+\nu(X,\theta(Y,Z))+(\ad_X^\h)^*\theta_2(Y,Z)\\
		(X\bullet Y)\bullet Z&=(\e[X, Y]_\h+\theta(X,Y)+\theta_2(X,Y))\bullet Z\\
		&=\e[[X,Y]_\h,Z]_\h+\e\theta([X,Y]_\h,Z)+\e\theta_2([X,Y]_\h,Z)+G(Z)\theta(X,Y)+\rho(\theta(X,Y),Z)\\
		Y\bullet(X\bullet Z)
		&=\e[Y,[X,Z]_\h]_\h+\e\theta(Y,[X,Z]_\h)
		+\e\theta_2(Y,[X,Z]_\h)+F(Y)\theta(X,Z)
		+\nu(Y,\theta(X,Z))+(\ad_Y^\h)^*\theta_2(X,Z).
	\end{align*}
$\bullet$ If the metric is $\Li$-quadratic, $Q(X,Y,Z)=0$ is equivalent to 	 $(\h,\br_\h)$ is a Lie algebra and 
\[ \begin{cases}
	\theta(X,[Y,Z]_\h)
	+F(X)\theta(Y,Z)-\theta([X,Y]_\h,Z)-G(Z)\theta(X,Y)-\theta(Y,[X,Z]_\h)
	-F(Y)\theta(X,Z)=0,\\
	K:=\theta_2(X,[Y,Z]_\h)+\nu(X,\theta(Y,Z))+(\ad_X^\h)^*\theta_2(Y,Z)-\theta_2([X,Y]_\h,Z)-\rho(\theta(X,Y),Z)-\theta_2(Y,[X,Z]_\h)\\
	-\nu(Y,\theta(X,Z))-(\ad_Y^\h)^*\theta_2(X,Z)=0.
\end{cases} \]
We get the equation (5) in \eqref{eqmain3}. We have
\begin{align*}
	\prec K,T\succ&=\Om(X)([Y,Z]_\h,T)-\langle\theta(X,T),\theta(Y,Z)\rangle_A-\Om(Y)(Z,[X,T]_\h)-\Om([X,Y]_\h)(Z,T)\\&+\langle\om(Z,T),\theta(X,Y)\rangle_A-\Om(Y)([X,Z]_\h,T)+\langle\theta(Y,T),\theta(X,Z)\rangle_A
	+\Om(X)(Z,[Y,T]_\h)\\
	&=\De(\Om)(X,Y)(Z,T)-\langle\theta(X,T),\theta(Y,Z)\rangle_A+\langle\om(Z,T),\theta(X,Y)\rangle_A+\langle\theta(Y,T),\theta(X,Z)\rangle_A.
\end{align*}
So $K=0$ if and only if the last equation in \eqref{eqmain3} holds.
	
	$\bullet$ If the metric is $\Ri$-quadratic, $Q(X,Y,Z)=0$ is equivalent to 	 
	\[ \begin{cases}
		F(X)\theta(Y,Z)+F(Y)\theta(Z,X)+F(Z)\theta(X,Y)=0,\\
		K:=\nu(X,\theta(Y,Z))+(\ad_X^\h)^*\theta_2(Y,Z)-
		\rho(\theta(X,Y),Z)
		-\nu(Y,\theta(X,Z))-(\ad_Y^\h)^*\theta_2(X,Z)=0.
		\end{cases}
		 \]
We get the equation (3) in \eqref{eqmain4}. 	We have
\begin{align*}
	\prec K,T\succ&=\langle\om(X,T),\theta(Y,Z)\rangle_A-\Om(Z)(Y,[X,T]_\h)-\langle\theta(Z,T),\theta(X,Y)\rangle_A-\langle\om(Y,T),\theta(X,Z)\rangle_A+\Om(Z)(X,[Y,T]_\h)\\
	&=(\ad_T^\h)^*\Om(Z)(X,Y)+\langle\om(X,T),\theta(Y,Z)\rangle_A-\langle\theta(Z,T),\theta(X,Y)\rangle_A-\langle\om(Y,T),\theta(X,Z)\rangle_A.
\end{align*}
So $K=0$ if and only if the last equation in \eqref{eqmain4} holds.

	\item We have    
	\begin{align*}
		X\bullet (Y\bullet \al)&= X\bullet (\mathrm{ad}^{\h}_Y)^* \al\\
		&= (\mathrm{ad}^{\h}_X)^*(\mathrm{ad}^{\h}_Y)^* \al \\
		(X\bullet Y)\bullet \al&=(\e[X,Y]_{\h}+ \theta(X,Y)+\theta_{2}(X,Y))\bullet  \al\\
		&= (\mathrm{ad}^{\h}_{\e[X,Y]})^* \al \\
		Y\bullet (X\bullet \al)&= (\mathrm{ad}^{\h}_Y)^*(\mathrm{ad}^{\h}_X)^* \al).
	\end{align*}
$\bullet$ If the metric is $\Li$-quadratic then $Q(X,Y,\al)=0$ if and only if $(\h,\br_\h)$ is a Lie algebra.

$\bullet$ If the metric is $\Ri$-quadratic then $Q(X,Y,\al)=0$ if and only if $(\h,\br_\h)$ is a Lie algebra and $[X,[Y,Z]_\h]_\h=0$ for any $X,Y,Z\in\h$.

	\item  
	We have
	\begin{align*}
		X\bullet (Y\bullet a)&=X\bullet( F(Y)a+ \nu(Y,a))\\
		&= F(X)F(Y)a+ \nu(X, F(Y)a)+ (\mathrm{ad}^{\h}_{X})^*\nu(Y,a)\\
		(X\bullet Y)\bullet a &= (\e[X,Y]_{\h}+ \theta(X,Y)+\theta_{2}(X,Y))\bullet a\\
		&= F(\e[X,Y]_{\h})a+ \nu(\e[X,Y]_{\h},a)+ \theta(X,Y)\bullet_A a+ \mu(\theta(X,Y),a)\\
		Y\bullet(X\bullet a) &= Y\bullet (F(X)a+ \nu(X,a))\\
		&= F(Y)F(X)a+ \nu(Y, F(X)a)+ (\mathrm{ad}_{Y})^*\nu(X,a).
	\end{align*}
	
	$\bullet$ If the metric is $\Li$-quadratic then $Q(X,Y,a)=0$ if and only if $\langle Q(X,Y,a),Z\rangle_n=-\langle Q(X,Y,Z),a\rangle_n=0$ and
	\[ \Li_{\theta(X,Y)}=[F(X,F(Y))]-F([X,Y]_\h). \]
	The first condition is satisfied in the previous case and the second one is equivalent to a part of the equation (3) in \eqref{eqmain3}.
	
	$\bullet$ if the metric is $\Ri$-quadratic then 
	$Q(X,Y,a)=0$ if and only if
	\[ \begin{cases}
		\ad_{\theta(X,Y)}^A=[F(X),F(Y)],\\
		K:=\nu(X, F(Y)a)+ (\mathrm{ad}^{\h}_{X})^*\nu(Y,a)-\mu(\theta(X,Y),a)-\nu(Y, F(X)a)- (\mathrm{ad}_{Y})^*\nu(X,a)=0.
	\end{cases} \]

The first relation is equivalent to  the equation (2) in \eqref{eqmain4}. We have
\begin{align*}
	\prec K,Z\succ&=\langle\om(X,Z),F(Y)a\rangle_A-
	\langle\om(Y,[X,Z]_\h),a\rangle_A-\langle F(Z)\theta(X,Y),a\rangle_A-\langle\om(Y,Z),F(X)a\rangle_A+\langle\om(X,[Y,Z]_\h),a\rangle_A.
\end{align*} So $K=0$ if and only if
\[ F(Y)\om(X,Z)+\om(Y,[X,Z]_\h) +F(Z)\theta(X,Y)-F(X)\om(Y,Z)-\om(X,[Y,Z]_\h)=0.\]This is equivalent to the equation (4) in \eqref{eqmain4}.

	\item We have   
	\begin{align*}
		X\bullet (a\bullet Y)&= X\bullet (G(Y)a+\rho(a,Y))\\
		&= F(X)G(Y)a+\nu(X,G(Y)a )+(\mathrm{ad}^{\h}_X)^*(\rho(a,Y)),\\
		(X\bullet a)\bullet Y&=(F(X)a+\nu(X,a))\bullet Y\\
		&=G(Y)F(X)a+\rho(F(X)a,Y),\\
		a\bullet(X\bullet Y)&=\e G([X,Y]_\h)a+\rho(a,\e [X,Y]_\h)+a\bullet_A \theta(X,Y)+\mu(a,\theta(X,Y)).
	\end{align*}
	
	$\bullet$ If the metric is $\Li$-quadratic then $Q(X,a,Y)=0$ if and only if
	\[ \begin{cases}
		\Ri_{\theta(X,Y)}=[F(X),G(Y)]-G([X,Y]_\h)\\
		K:=\nu(X,G(Y)a )+(\mathrm{ad}^{\h}_X)^*(\rho(a,Y))-\rho(F(X)a,Y)-\rho(a, [X,Y]_\h)-\mu(a,\theta(X,Y))=0.
	\end{cases} \]
We have
\begin{align*}
	\prec K,Z\succ&=-\langle \theta(X,Z),G(Y)a\rangle_A
	+\langle\om(Y,[X,Z]_\h),a\rangle_A+\langle\om(Y,Z),F(X)a\rangle_A+\om([X,Y]_\h,Z),a\rangle_A+\langle G(Z)a,\theta(X,Y)\rangle_A.
\end{align*}
So $Q(X,a,Y)=0$ if and only if the equation (3) and equation (6) in \eqref{eqmain3} hold.

$\bullet$ If the metric is $\Ri$-invariant $Q(X,a,Y)=0$ if and only if
\[ \begin{cases}
	\ad_{\theta(X,Y)}^A=[F(X),F(Y)]\\
	K:=-\nu(X,F(Y)a )+(\mathrm{ad}^{\h}_X)^*(\rho(a,Y))-\rho(F(X)a,Y)-\mu(a,\theta(X,Y))=0.
	\end{cases} \]
We have
\begin{align*}
	\langle K,Z\succ&=-\langle\om(X,Z),F(Y)a\rangle_A
	-\langle\theta(Y,[X,Z]_\h),a\rangle_A-\langle\om(Y,Z),F(X)a\rangle_A-\langle F(Z)a,\theta(X,Y)\rangle.
\end{align*}So $K=0$ if and only if
\[ F(Y)\om(X,Z)-\theta(Y,[X,Z]_\h)+F(X)\theta(Y,Z)+F(Z)\theta(X,Y)=0. \]
So $Q(X,a,Y)=0$ if and only if the equations (2) and (5) in \eqref{eqmain4} hold.

	\item We have   
	\begin{align*}
		X\bullet (a\bullet b)&= X\bullet( a\bullet_A b+ \mu(a,b))\\
		&=F(X)(a\bullet_A b)+\nu(X,a\bullet_A b)+ (\mathrm{ad}^{\h}_X)^*(\mu(a,b))\\
		(X\bullet a)\bullet b&= (F(X)a+ \nu(X,a))\bullet b\\
		&=F(X)a\bullet_A b+\mu(F(X)a,b)\\
		a\bullet (X\bullet b)&= a\bullet (F(X)b+ \nu(X,b) )\\
		&= a\bullet_A F(X)b+ \mu(a,F(X)b).
	\end{align*}
	
	$\bullet$ If the metric is $\Li$-quadratic then $Q(X,a,b)=0$ if and only if $\langle Q(X,a,b),Y\rangle_n=-\langle Q(X,a,Y),b\rangle_n=0$ and $F(X)$ is a derivation for any $X$. We have already get the condition for $\langle Q(X,a,Y),b\rangle_n=0$.
	
	$\bullet$ If the metric is $\Ri$-quadratic then $Q(X,a,b)=0$ if and only if $F(X)$ is a derivation for any $X$ and
	\[ K:= \nu(X,a\bullet_A b)+ (\mathrm{ad}^{\h}_X)^*(\mu(a,b))-\mu(F(X)a,b)
	-\mu(a,F(X)b)=0.\] 
	We have
	\begin{align*}
		\prec K,Y\succ&=\langle\om(X,Y),[a,b]_A\rangle_A
		-\langle F([X,Y]_\h)a,b\rangle -\langle F(Y)F(X)a,b\rangle-\langle_A F(Y)a,F(X)b\rangle_A. 
		\end{align*}
	So $K=0$ is equivalent to $\ad^A_{\om(X,Y)}=[F(X),F(y)]-F([X,Y]_\h)$ which is equivalent the equation (1) in \eqref{eqmain4}.
	
	\item  We have 
	\begin{align*}
		a\bullet (b\bullet  X)&=a\bullet ( G(X)b+\rho(b,X))\\
		&=a\bullet_A G(X)b + \mu(a, G(X)b))\\
		(a\bullet b)\bullet X&=(a\bullet_Ab+\mu(a,b)) \bullet X\\
		&= G(X)(a\bullet_Ab)+\rho(a\bullet_A b,X)\\
		b\bullet (a\bullet X)&=b\bullet( G(X)a+\rho(a,X))\\
		&=b\bullet_AG(X)a + \mu(b, G(X)a).
	\end{align*}

$\bullet$ If the metric is $\Li$-quadratic then $Q(a,b,X)=0$ if and only if
\[ \begin{cases}
	G(X)(a\bullet_Ab)=a\bullet_A G(X)b-b\bullet_AG(X)a,\\
	K:=\mu(a, G(X)b))-\rho(a\bullet_A b,X)-\mu(b, G(X)a)=0.
	\end{cases}
 \]
 
 We have
 \begin{align*}
 	\prec K,Y\succ&=-\langle G(Y)a,G(X)b\rangle_A+\langle \om(X,Y),a\bullet_A b\rangle_A+\langle G(Y)b,G(X)a\rangle_A.
 \end{align*}
So $K=0$ is equivalent to
\[ \Li_{\om(X,Y)}^A=G(X)^*G(Y)-G(Y)^*G(X). \]
So $Q(a,b,X)=0$ if and only if the second part of the first equations and  the second equation in \eqref{eqmain3} hold..

$\bullet$ If the metric is $\Ri$-quadratic  gives that $Q(a,b,X)=0$ if and only if $F$ is a derivation an $K=0$ where
\[ K=-\mu(a, F(X)b))-\rho([a,b]_A,X)+\mu(b, F(X)a)=0 \]
\begin{align*}
	\prec K,Y\succ&=-\langle F(Y)a,F(X)b\rangle_A-\langle \theta(X,Y),[a, b]_A\rangle_A+\langle F(Y)b,F(X)a\rangle_A.
\end{align*}So $K=0$ if and only if the second equation in \eqref{eqmain4} holds.

	\item We have    
	\begin{align*}
		a \bullet (b\bullet c)&=a\bullet (b\bullet_A c+ \mu(b,c))\\
		&=a\bullet_A(b\bullet_A c)+\mu(a,b\bullet_A c)\\
		(a\bullet b)\bullet c&=(a\bullet_A b+\mu(a,b))\bullet  c\\
		&=(a\bullet_A b)\bullet_A c+\mu(a\bullet_A b,c)\\
		b\bullet (a\bullet c)&=b\bullet (a\bullet_A c+ \mu(a,c))\\
		&=b\bullet_A(a\bullet_A c)+\mu(b,a\bullet_A c).
	\end{align*}
	$\bullet$ If the metric is $\Li$-quadratic then $Q(a,b,c)=0$ is equivalent to $\langle Q(a,b,c),X\rangle_n=-\langle Q(a,b,X),c\rangle_n=0$ and $(A,\bullet_A)$ is a left Leibniz algebra.
	
	$\bullet$ If the metric is $\Ri$-quadratic then $Q(a,b,c)=0$ is equivalent to $(A,\br_A)$ is a Lie algebra and
	\[ K=\mu(a,[b, c]_A)-\mu([a, b]_A,c)-\mu(b,[a c]_A)=0. \]
	We have
	\begin{align*}
		\prec K,X\succ&=\langle F(X)a,[b, c]_A\rangle_A
		-\langle F(X)[a,b]_A,c\rangle_A-\langle F(X)b,[a,c]_A\rangle_A
	\end{align*}which is equivalent to $F(X)$ is a derivation.
	
	\item We have
	\begin{align*}
		a\bullet (X\bullet Y)&=a\bullet (\e[X,Y]_{\h}+ \theta(X,Y)+\theta_{2}(X,Y))\\
		&=G(\e[X,Y]_{\h})a+\rho(a,\e[X,Y]_{\h})+ a\bullet_A\theta(X,Y)+ \mu(a,\theta(X,Y))\\
		(a\bullet X)\bullet Y&=(G(X)a+\rho(a,X))\bullet Y\\
		&= G(Y)G(X)a+ \rho(G(X)a, Y),\\
		X\bullet (a\bullet Y)&=X\bullet (G(Y)a+\rho(a,Y))\\
		&= F(X)G(Y)a+\nu(X, G(Y)a) +(\mathrm{ad}^{\h}_X)^*(\rho(a,Y)).
	\end{align*}
	
	$\bullet$ if the metric is $\Li$-quadratic then $Q(a,X,Y)=0$ if and only if
	\[ \begin{cases}
		\Ri^A_{\theta(X,Y)}=F(X)G(Y)+G(Y)G(X)-G([X,Y]_{\h}),\\
		K:=\rho(a,[X,Y]_{\h})+\mu(a,\theta(X,Y))
		-\rho(G(X)a, Y)-\nu(X, G(Y)a)-(\mathrm{ad}^{\h}_X)^*(\rho(a,Y))=0.
	\end{cases} \]
	We have
	\begin{align*}
		\prec K,Z\succ&=-\langle \om([X,Y]_\h,Z),a\rangle_A-\langle G(Z)a,\theta(X,Y)\rangle_A+\langle \om(Y,Z),G(X)a\rangle_A+\langle \theta(X,Z),G(Y)a\rangle_A-\langle\om(Y,[X,Z]_\h),a\rangle_A.
	\end{align*}
So $K=0$ if and only if
\[ G(Z)^*\theta(X,Y)-G(Y)^*\theta(X,Z)-G(X)^*\om(Y,Z)+\om([X,Y]_\h,Z)+\om(Y,[X,Z]_\h)=0. \]
This equation resembles to equation (6) and by taking the difference, we have
 $Q(a,X,Y)=0$ if and only if the second part of equation (3) and the equation (7) in \eqref{eqmain3} hold.

$\bullet$ If the metric is $\Ri$-quadratic then
$Q(a,X,Y)=0$ if and only if
\[ \begin{cases}
	\ad^A_{\theta(X,Y)}=[F(X),F(Y)],\\
	K:=\mu(a,\theta(X,Y))
	+\rho(F(X)a, Y)+\nu(X, F(Y)a)-(\mathrm{ad}^{\h}_X)^*(\rho(a,Y))=0.
\end{cases} \]
We have
\begin{align*}
	\prec K,Z\succ&=\langle F(Z)a,\theta(X,Y)\rangle_A
	+\langle\theta(Y,Z),F(X)a\rangle_A+\langle\om(X,Z),F(Y)a\rangle_A+\langle \theta(Y,[X,Z]_\h),a\rangle_A.
\end{align*}So $K=0$ if and only if
\[ F(Z)\theta(X,Y)+F(X)\theta(Y,Z)-F(Y)\om(Z,X)-\theta(Y,[X,Z]_\h)=0. \]
Finally, $Q(a,X,Y)=0$ if and only if the equations (2) and (5) in \eqref{eqmain4} hold.

		\item We have 
	\begin{align*}
		a\bullet (X\bullet b)&=a\bullet( F(X)b+\nu(X,a))\\
		&=a\bullet_A F(X)b+\mu(a,F(X)b)\\
		(a\bullet X)\bullet b&=(G(X)a+\rho(a,X))\bullet b\\
		&= G(X)a\bullet_A b+\mu(G(X)a,b)\\
		X\bullet (a\bullet b)&=X\bullet( a\bullet_Ab+ \mu(a,b))\\
		&=F(X)(a\bullet_Ab)+\nu(X,a\bullet_Ab)+ (\mathrm{ad}^{\h}_X)^*(\mu(a,b)).
	\end{align*}

$\bullet$ If the metric is $\Li$-quadratic then $Q(a,X,b)=0$ if and only if $\langle Q(a,X,b),Y\rangle_n=-\langle Q(a,X,Y),b\rangle_n=0$ and

\[ F(X)(a\bullet_Ab)=a\bullet_A F(X)b-G(X)a\bullet_A b \]
which is equivalent to the first equation in \eqref{eqmain3}.

$\bullet$ If the metric is $\Ri$-quadratic then $Q(a,X,b)=0$ if and only if $F(X)$ is a derivation and
\[ K:=\mu(a,F(X)b)+\mu(F(X)a,b)-\nu(X,[a,b]_A)- (\mathrm{ad}^{\h}_X)^*(\mu(a,b))=0. \]
We have
\begin{align*}
	\prec K,Y\succ&=\langle F(Y)a,F(X)b\rangle_A+\langle F(Y)F(X)a,b\rangle_A
	-\langle\om(X,Y),[a,b]_A+\langle F([X,Y]_\h)a,b\rangle_A.
	\end{align*}So $K=0$ is equivalent to
\[ -F(X)F(Y)+F(Y)F(X)-\ad^A_{\om(X,Y)}+F([X,Y]_\h)=0 \]
which is equivalent to the first equation in \eqref{eqmain4}.

	\item $Q(X,a,\al)=Q(a,X,\al)= Q(a,b,\al)= 0$ are satisfied.\end{enumerate}

\section{ Low dimensional metrised, $\Li$-quadratic or $\Ri$-quadratic left Leibniz algebras}\label{section8}

We illustrate all our results by giving the list of metrised symmetric Leibniz algebras of dimension 4,5 and 6 obtained from indecomposable non abelian quadratic Lie algebras, the list of $\Li$-quadratic non Lie left Leibniz algebras of $\leq$ 4 and the list of $\Ri$-quadratic non Lie Leibniz algebras of dimension $\leq$ 5.

{\footnotesize
	
	{\renewcommand*{\arraystretch}{1.4}
		\begin{center}	
			\begin{tabular}{|c|c|c|}
				\hline
				Algebra&The non vanishing bracket&The metric\\
				\hline
				$\mathrm{os}(4,\la)$&
				$[e_2,e_3]=e_1,[e_4,e_2]=\la e_3,[e_4,e_3]=-\la e_2,e_4\circ e_4=\mu e_1$&$\{(2,2)=(3,3)=\frac1{\la},(1,4)=1\}$\\
				\hline
				$\G_{1,4}$&$[e_4,e_2]=e_2,\;[e_4,e_3]=-e_3,\;[e_2,e_3]=e_1,e_4\circ e_4=\mu e_1$&$\{(1,4)=(2,3)=1\}$\\
				\hline
				$\G_{1,5}$&$[e_2,e_3]=e_1,[e_3,e_4]=-e_1,[e_5,e_2]=e_3,$&$\{(1,5)=-(2,2)=(3,3)=(4,4)=1\}$\\
				&$[e_5,e_3]=e_2-e_4,[e_5,e_4]=e_3,e_5\circ e_5=\mu e_1$&\\
				
				\hline
				$\mathrm{osc}(6,(\la_1,\la_2))$&
				$[e_2,e_4]=[e_3,e_5]=e_1,[e_6,e_2]=\la_1 e_4$&$\{(2,2)=(3,3)=(4,4)=(5,5)=\frac1{\la} $\\
				&$[e_6,e_4]=-\la_1e_2,[e_6,e_3]=\la_2e_5,[e_6,e_5]=-\la_2e_3$,$e_6\circ e_6=\mu e_1$&$(1,6)=1\}$\\		
				\hline	
				$\mathfrak{n}_1(2,2)$&$[e_6,e_3]=e_2,[e_6,e_5]=e_4,[e_3,e_5]=e_1.$&$\{(1,6)=(2,5)=-(3,4)=1\}$\\
				\hline
				$\mathfrak{n}_2(2,2)$&
				$[e_6,e_2]=e_2+te_3,[e_6,e_3]=-te_2+e_3,[e_6,e_4]=-e_4+te_5$&$\{(1,6)=(2,5)=-(3,4)=1\}$\\
				&$[e_6,e_5]=-te_4-e_3, [e_2,e_4]=-te_1,[e_2,e_5]=e_1,$&\\	
				&$[e_3,e_4]=-e_1,[e_3,e_5]=-te_1,\;e_6\circ e_6=\mu e_1\quad t>0$&\\
				\hline
				$\mathfrak{n}_3(2,2)$&$[e_6,e_2]=e_3,[e_6,e_3]=-e_2,[e_6,e_4]=\epsilon e_2+e_5,\e^2=1,e_6\circ e_6=\mu e_1$&$	\{(1,6)=(2,5)=-(3,4)=1\}$\\
				&$[e_6,e_5]=\epsilon e_3-e_4, [e_2,e_4]=-e_1,[e_3,e_5]=-e_1,[e_4,e_5]=\e e_1.$&\\	
				\hline
				$\mathfrak{n}_4(2,2)$&$	[e_6,e_2]=e_3,[e_6,e_3]=-e_2,[e_6,e_4]=te_5,e_6\circ e_6=\mu e_1$&$\{-(2,2)=-(3,3)=(4,4)=(5,5)=1  $\\	
				&$[e_6,e_5]=-te_4,  [e_2,e_3]=-e_1,[e_4,e_5]=te_1.\quad (t>0)$&
				$(1,6)=1\}$\\
				\hline
				$\mathfrak{n}_5(2,2)$&$[e_6,e_2]=e_2,[e_6,e_3]=e_2+e_3,[e_6,e_4]=-e_4,[e_6,e_5]=e_4-e_5,$&$\{(1,6)=(2,5)=-(3,4)=1\}$\\
				&$[e_2,e_5]=e_1,[e_3,e_4]=-e_1,[e_3,e_5]=e_1,e_6\circ e_6=\mu e_1.$&\\
				\hline
				$\mathfrak{n}_6(2,2)$&$[e_6,e_2]=e_2,[e_6,e_3]=-e_3,[e_6,e_4]=te_4,e_6\circ e_6=\mu e_1$&$\{(1,6)=(2,3)=(4,5)=1  \}$\\
				&$[e_6,e_5]=-te_5,  [e_2,e_3]=e_1,[e_4,e_5]=te_1,t\geq1$&\\
				\hline
			\end{tabular}
			\captionof{table}{Metrised symmetric Leibniz algebras of dimension 4,5 and 6 obtained from indecomposable non abelian quadratic Lie algebras $(\mu\not=0)$.\label{1}}
\end{center}}}

{\renewcommand*{\arraystretch}{1.4}
	\begin{center}	
		\begin{tabular}{|c|c|c|}
			\hline
			Algebra&The non vanishing bracket&The metric\\
			\hline
			$\g_{3,1}$&
			$e_3\bullet e_1=\la e_2,\; e_3\bullet e_2=-\la e_1$, $\la \neq 0$&$\{(1,1)=(2,2)=1,(3,3)=\pm1\}$\\
			\hline
			$\g_{3,2}$	&$e_3\bullet e_1=\la e_2,\; e_3\bullet e_2=\la e_1$, $\la \neq 0$&$\{(1,1)=-1,(2,2)=(3,3)=1\}$\\
			\hline
			$\g_{3,3}$	&$e_2\bullet e_2 =-\mu e_3, \; e_2\bullet e_3=\mu e_1, \mu \neq 0$&$\{(1,2)=(2,1)=(3,3)=1\}$\\
			\hline
			$\g_{4,1}$&$e_1\bullet e_3= e_4,\; e_1\bullet e_4=-e_3,\; e_2\bullet e_3=e_4,$&$\{(1,1)= b,(2,2)=(3,3)=1,(4,4)=a \}, $\\
			&$e_2\bullet e_4= -e_3$&$a^2=b^2=1$\\		
			\hline	
			$\g_{4,2}$	&$e_1\bullet e_1=ae_4,\; e_1\bullet e_2=-ae_3,\; e_2\bullet e_1=be_4, \; e_2\bullet e_2=-be_3$&$\{(1,3)=(3,1)=(2,4)=(4,2)=1\}$\\
			\hline
			$\g_{4,3}$	&$e_1\bullet e_2=e_1,\; e_2 \bullet e_1=-e_1+ae_4,\; e_2\bullet e_2=-ae_3,$&$\{(1,3)=(3,1)=(2,4)=(4,2)=1\}$\\
			&$e_1\bullet e_3=-e_4, \;e_2\bullet e_3=e_3$&\\
			\hline
			$\g_{4,4}$	&$e_4\bullet e_4=-\al e_2,\;
			e_4\bullet e_2=\la e_3+\al e_1,\; e_4\bullet e_3=-\la e_2$& $	\{(1,6)=(2,5)=-(3,4)=1\}$\\	
			\hline
			$\g_{4,5}$		&$e_2\bullet e_3 =\la e_1,\;e_3\bullet e_2=-\la e_1,\; e_4\bullet e_2=\la e_3,\; e_4\bullet e_3=-\la e_2 ,$ &$\{(1,6)=(2,5)=-(3,4)=1\}$\\
			&$ e_2\bullet e_4=-\la e_3,\;e_3\bullet e_4=\la e_2$&\\
			\hline
			$\g_{4,6}$&$e_4\bullet e_4=-\al e_2-\be e_3
			,\;e_3\bullet e_2=\ga e_1,e_4\bullet e_2=\al e_1,\;$&$	\{(1,6)=(2,5)=-(3,4)=1\}$\\
			&$ e_4\bullet e_3=\be e_1,\;e_3\bullet e_4=-\ga e_2,\quad \be \ga=0$&\\	
			\hline
		\end{tabular}
		\captionof{table}{$\Li$-quadratic non Lie left Leibniz algebras of dimension $\leq$ 4.\label{2}}
\end{center}}

{\renewcommand*{\arraystretch}{1.4}
	\begin{center}	
		\begin{tabular}{|c|c|c|}
			\hline
			Algebra&The non vanishing products&The metric\\
			\hline
			$\Li_{1,4}$&$e_1\bullet e_1=-\la_1e_4,
			e_1\bullet e_2=-\la_2e_4,
			e_2\bullet e_1=\la_1e_3,
			e_2\bullet e_1=\la_2e_3$&$\{(1,3)=(2,4)=1 \}$\\
			\hline
			$\Li_{1,5}$& $e_1\bullet e_1=-\la_1 e_5,\;e_1\bullet e_2=-\la_2 e_5+\rho e_3,e_2\bullet e_1=\la_1 e_4-\rho e_3$&$\{(1,4)=(2,5)=(3,3)=1\}$\\
			&$e_2\bullet e_2=\la_2 e_4,
			e_1\bullet e_3=\mu e_5, e_2\bullet e_3=-\mu e_4, e_3\bullet e_1=\rho e_5, e_3\bullet e_2=-\rho e_4$&$\rho=0$ or $\rho=-\mu$\\
			\hline		
			
		\end{tabular}
		\captionof{table}{$\Ri$-quadratic non Lie Leibniz algebras of dimension $\leq$ 5. .\label{3}}
\end{center}}

\end{document}